%% file: main_arXiv.tex
\documentclass[journal]{IEEEtran}

\IEEEoverridecommandlockouts                              
\usepackage{cite}
\usepackage{enumitem}
\usepackage{amsmath,amssymb,mathrsfs,amsthm,amsfonts}
\newtheorem{thm}{Theorem}[]
\newtheorem{prop}[thm]{Proposition}
\newtheorem{lemma}[thm]{Lemma}
\newtheorem{asmn}[thm]{Assumption}
\usepackage{url}
\usepackage{tabularx}
\usepackage{array}
\newcolumntype{L}{>{\centering\arraybackslash}m{1.0cm}}
\usepackage[ruled,linesnumberedhidden,titlenotnumbered]{algorithm2e}
\usepackage{algpseudocode}

\usepackage[pdftex]{graphicx}
\DeclareGraphicsExtensions{.pdf,.jpeg,.png,.JPG,.eps,.tif}
\graphicspath{{figs/}}

\usepackage{epstopdf}
\usepackage{booktabs}
\usepackage{comment}

\usepackage{multirow}
\usepackage{multicol}
\usepackage{tabularx}

\usepackage{enumitem}
\usepackage{wrapfig}
\usepackage{subcaption}

\usepackage{hyperref}

\usepackage{cleveref}
\ifCLASSOPTIONcompsoc
\else
\fi

\usepackage{color}

\usepackage[normalem]{ulem}

\definecolor{darkgrn}{rgb}{0, 0.8, 0}

\newcommand{\txbl}[1]{\textcolor{blue}{#1}}

\newcommand{\A}{\mathscr{A}}
\newcommand{\B}{\mathscr{B}}
\newcommand{\bphi}{\boldsymbol{\varphi}}
\newcommand{\bla}{\boldsymbol{\lambda}}
\newcommand{\bmu}{\boldsymbol{\mu}}

\newcommand{\bz}{\mathbf{z}_j}

\newcommand{\bzb}{\overline{\mathbf{z}}_j}
\newcommand{\p}{P_{ij}}
\newcommand{\q}{Q_{ij}}
\newcommand{\ri}{r_{ij}}
\newcommand{\x}{x_{ij}}
\newcommand{\vi}{v_i}

\newcommand{\bc}{\bar{c}}
\newcommand{\bzero}{\mathbf{0}}

\newcommand{\E}{\mathcal{E}}
\newcommand{\N}{\mathcal{N}}

\newcommand{\D}{\Delta}
\newcommand{\Dt}{\D^{(t)}}

\newcommand{\pt}[1]{P^{(t)}_{#1}}

\newcommand{\pa}[1]{P^{(t+1)}_{#1}}
\newcommand{\qt}[1]{Q^{(t)}_{#1}}

\newcommand{\qa}[1]{Q^{(t+1)}_{#1}}
\newcommand{\lt}[1]{\ell^{(t)}_{#1}}
\newcommand{\la}[1]{\ell^{(t+1)}_{#1}}

\newcommand{\rt}[1]{r_{#1}}
\newcommand{\xt}[1]{x_{#1}}
\newcommand{\vt}[1]{v^{(t)}_{#1}}
\newcommand{\va}[1]{v^{(t+1)}_{#1}}
\newcommand{\vm}[1]{v^{(t-1)}_{#1}}

\newcommand{\pti}[1]{P^{(t_i)}_{#1}}
\newcommand{\pai}[1]{P^{(t_i+1)}_{#1}}
\newcommand{\qti}[1]{Q^{(t_i)}_{#1}}

\newcommand{\rti}[1]{r_{#1}}
\newcommand{\xti}[1]{x_{#1}}
\newcommand{\vmi}[1]{v^{(t_i-1)}_{#1}}

\newcommand{\imi}{\tilde{j}} 

\begin{document}

\title{Convergence Guarantees of a Distributed Network Equivalence Algorithm for Distribution-OPF}

\author{Yunqi Luo, Rabayet Sadnan,~\IEEEmembership{Student Member,~IEEE}, Bala Krishnamoorthy, and~Anamika Dubey,~\IEEEmembership{Senior Member,~IEEE}
\vspace*{-0.2in}
\thanks{YL is with University of Science and Technology of China, RS and AD are with EECS, Washington State University, Pullman and, BK is with  Mathematics and Statistics, Washington State University, Vancouver. YL and BK acknowledge funding from NSF through grant 1819229. RS and AD acknowledge funding from DOE under contract DE-AC05-76RL01830.
 Corresponding author E-mail: anamika.dubey@wsu.edu.}}
\maketitle

\begin{abstract}
  Optimization-based approaches have been proposed to handle the integration of distributed energy resources into the electric power distribution system. 
  The added computational complexities of the resulting optimal power flow (OPF) problem have been managed by approximated or relaxed models; but they may lead to infeasible or inaccurate solutions.
  Other approaches based on decomposition-based methods require several message-passing rounds for relatively small systems, causing significant delays in decision-making.
  We propose a distributed algorithm with convergence guarantees called ENDiCo-OPF for nonlinear OPF.
  Our method is based on a previously developed decomposition-based optimization method that employs network equivalence.
  We derive a sufficient condition under which ENDiCo-OPF is guaranteed to converge for a single iteration step on a local subsystem.
  We then derive conditions that guarantee convergence of a local subsystem over time.
  Finally, we derive conditions under a suitable assumption that when satisfied in a time sequential manner guarantee global convergence of a \emph{line network} in a sequential manner.
  We also present simulations using the IEEE-123 bus test system to demonstrate the algorithm's effectiveness and provide additional insights into theoretical results.
\end{abstract}


\IEEEpeerreviewmaketitle
\section*{Nomenclature}
\addcontentsline{toc}{section}{Nomenclature}
\begin{IEEEdescription}[\IEEEusemathlabelsep\IEEEsetlabelwidth{$V_1,V_2,V_3$}]
\item[\textit{Sets}]
\item[$\N$] Set of all nodes 
\item[$\E$] Set of all distribution lines
\item[$\N_D$] Set of all DER nodes in the system
\item[$\N_L$] Set of all load nodes in the system
\vspace{0.1cm}
\item[\textit{Functions}]
\item[\textbf{$f$}] Cost function for the OPF sub-problem
\item[\textbf{$F$}] Function describing the sub-problem
\item[$\A$] Equality constraint function
\item[$\B$] Inequality constraint function
\item[$L$] Lagrangian function
\vspace{0.1cm}
\item[\textit{Variables}]   
\item[$x_{ij}$] Line reactance for line $(i,j)$
\item[$r_{ij}$] Line resistance for line $(i,j)$
\item[$V_i$] Voltage magnitude for node $i$
\item[$v_i = |V_i|^2$] Squared voltage magnitude for node $i$
\item[$l_{ij}$] Squared magnitude of the current flow in line $(i,j)$
\item[$P_{ij}$] Sending end active power for line $(i,j)$
\item[$Q_{ij}$] Sending end reactive power for line $(i,j)$
\item[$p_{L_j}$] Active load demand at node $j$
\item[$q_{L_j}$] Reactive load demand at node $j$
\item[$p_{Dj}$] Active power output by DER at node $j$
\item[$q_{Dj}$] Reactive power output by DER at node $j$
\item[$S_{DRj}$] kVA rating of the DER at node $j$
\item[$I_{ij}^{\text{rated}}$] Thermal rating for the line $(i,j)$
\item[$\bz$] Vector of local variables at node $j$
\item[$R$] Cost coefficient vector
\item[$\lambda$] Dual variable for equality constraints of sub-problems
\item[$\mu$] Dual variable for inequality constraints of sub-problems
\item[$\varphi$] Auxiliary variable to convert inequality constraint to the equation
\item[$\Delta^{(\cdot)}\geq 1$] Convergence parameter
\end{IEEEdescription}

\vspace{-0.3cm}
\section{Introduction}
\label{Introduction}

The nature and the requirements of power systems, especially at the distribution level, are rapidly changing with the large-scale integration of controllable distributed energy resources (DERs). Optimal power flow (OPF) methods have emerged as a possible mechanism to coordinate DERs for specified grid services \cite{castillo2013survey,patari2021distributed,lin2019decentralized,molzahn2019survey,molzahn2017survey}. However, the nonconvex power flow constraints in the distribution OPF (D-OPF) problem poses significant computational challenges that increase drastically with the size of the distribution systems \cite{molzahn2017survey,molzahn2019survey}. The OPF problems involve nonconvex power flow equations and are generally NP-hard problems, even for the radial power distribution networks \cite{molzahn2019survey}. There is an extensive body of literature highlighting the computational challenges of OPF problems that worsens with the system size \cite{avramidis2021practical, usman2022three,jha2022distribution}. Existing methods manage the computational challenges using convex relaxation or linear approximation techniques that may lead to infeasible power flow solutions or high optimality gap \cite{gan2014convex,JhaPV}. 
The existing literature also includes successive linear programming and convex iterations techniques to address the scalability of D-OPF problems, while still achieving feasible and optimal power flow solutions \cite{castillo2015successive,watson2014current ,Jha2021,wang2017chordal,zamzam2016beyond}. However, they have not been shown to scale beyond a medium size feeder.

\begin{table*}[htp!]
\begin{minipage}{\textwidth}
\centering
\caption{Qualitative Comparison of available Distributed OPFs}\label{table:qual_comparison_opfs}
\begin{tabular}{ p{0.5 in}  p{1.9in} p{0.5 in} p{0.4in} p{0.6in}  p{1.2in} p{0.8in}}
  \toprule
  & \emph{Summary} & \emph{Penalty Term} & \emph{Workflow} & \emph{Decomposition Method} & \emph{Shared Variables} & \emph{Global Variable}\\
  \midrule 
  {ADMM \cite{boyd2011distributed,kim2000comparison}}   &  {Augmented Lagrangian is solved (with added penalty term to minimize computed shared boundary variables) for the decomposed sub-problems} & {quadratic} & {Sequential} & {Dual decomposition} & {Voltage magnitude, and active and reactive power flow} & {Is needed} \\
    \midrule
  {Proximal Message Passing (PMP) \cite{molzahn2017survey,erseghe2014distributed}}  &  {Augmented Lagrangian is solved (with added penalty term to minimize computed shared boundary variables) for the decomposed sub-problems} & {quadratic} & {Parallel} & {Dual decomposition} & {Voltage magnitude} &{Not needed} \\
    \midrule
  {ALADIN \cite{peng2016distributed}}  &  {Augmented Lagrangian is solved for the decomposed sub-problems, however, primal and dual variable is updated by solving an additional quadratic consensus optimization problem for the boundaries using hessian and gradients} & {quadratic} & {Parallel} & {Dual decomposition} & {Voltage phasor (magnitude and angle), and active and reactive power flows} &{Not needed} \\
    \midrule
  {APP \cite{cohen1980auxiliary,baldick1999fast,kim2000comparison}}  &  {Augmented Lagrangian with added penalty terms is solved for the decomposed sub-problems} & {Linear} & {Parallel} & {Dual decomposition} & {Voltage magnitude and active and reactive power flows} &{Is needed} \\
  \midrule
  {Proposed Method (ENDiCo-OPF) }  &  {At every micro-iterations, decomposed sub-problems approximate the upstream and the downstream areas as a voltage source and fixed loads, respectively. The updated values of voltage and power flows (obtained from solving micro-iterations) are shared at each macro-iteration.} & {N/A} & {Parallel} & {Primal decomposition} & {Voltage magnitude and active and reactive power flows} &{Not needed} \\
  \bottomrule
\end{tabular}
\end{minipage}
\end{table*}

On the other hand, decomposition approaches based on the Augmented Lagrangian Method (ALM) and its variant, the Alternating Direction Method of Multipliers (ADMM), have been applied successfully to scale ACOPF problems for large feeders \cite{molzahn2017survey,boyd2011distributed,engelmann2018toward,gebbran2022multiperiod,inaolaji2022consensus}. In a series of early papers, Baldick et al. applied a linearized ALM to a regional decomposition of ACOPF \cite{KiBa1997,BaKiChLu1999,KiBa2000}. Peng and Low applied ADMM to certain convex relaxations of ACOPF on radial networks \cite{PeLo2014,PeLo2015,PeLo2016}. Computational efficiency of ADMM has been reported in practice for nonconvex ACOPF as well \cite{KoKiSo2005,ChKiHu2011,SuPhGh2013,Er2014,MhVeCh2019}, with convergence guarantees studied under certain technical assumptions \cite{Er2014,SuSu2021}. Along with computational advantages, the distributed methods can be used to coordinate the decisions of physically distributed agents, provide added robustness to single-point failure, and reduce communication overheads \cite{sadnan2021distributed}. Unfortunately, generic distributed optimization algorithms such as ADMM do not guarantee convergence for a general nonconvex optimization problem and may take a many message-passing rounds to converge to a local optimal solution. Specific to the D-OPF problem, the existing methods require a large number of message-passing rounds among the agents (on the order of $10^2$--$10^3$) to converge for a single-step optimization, which is not preferred from both distributed computing and distributed coordination standpoints \cite{erseghe2014distributed, dall2013distributed,millar2016smart,magnusson2015distributed}. When used for distributed coordination, many communication/message-passing rounds among distributed agents increases the time-of-convergence (ToC) and results in significant delays with decision-making. Some of these challenges are mitigated by distributed online controllers; however, they also take several time-steps to track the optimal decisions \cite{bolognani2014distributed, cavraro2017local,bernstein2019real,bastianello2020distributed,qu2019optimal,hu2019branch}.

To address these challenges, we recently proposed a distributed algorithm for the optimization of radial distribution systems based on the equivalence of networks principle \cite{sadnan2020real,sadnan2021distributed}. A qualitative comparison of the existing distributed algorithms with the proposed OPFs is included in Table \ref{table:qual_comparison_opfs}. The proposed approach solves the original non-convex OPF problem for power distribution systems using a novel decomposition technique that leverages the structure of the power flow problem. The primary distinction from the dual-decomposition approaches, is that the proposed method exploits the physics, i.e., the unique upstream/downstream relation among the power flow variables observed in a radial power distribution system. The use of problem structure in our distributed algorithm results in a significant reduction in the number of message-passing rounds needed to converge to an optimal solution by orders of magnitude ($\sim 10^2$). This results in significant advantages over generic application of distributed optimization techniques for distributed computing or distributed coordination in radial power distribution systems. However, our previous work requires solving a generic nonlinear optimization problem at each distributed node and does not provide any convergence guarantees.
 
The objective of this paper to develop a distributed optimization algorithm with convergence guarantees to solve D-OPF problems in a radial power distribution system.
Our decomposition approach is based on the structure of power flow problem in radial distribution systems and employs method of multipliers to solve the distributed subproblems exchanging specific power flow. Then, we present a comprehensive mathematical analysis on the convergence of the proposed approach and how it relates to the structural decomposition of the problem and problem-specific variables. Standard sufficiency conditions for optimality in nonlinear optimization \cite{Be1999} could be used to derive a set of conditions that guarantee convergence of local systems within a single iteration step. While the distributed nature stemming from the decomposition approach of our algorithm leads to its strong performances, the same setting poses considerable challenges to derive theoretical convergence guarantees over the entire network and also over time. As we employ a decomposition approach that solves local subsystems to optimality followed by communication rounds to achieve global convergence, we develop a similar strategy to derive guarantees for the same. To this end, we specify an additional condition on the convergence of voltage over time (Eqn.~\ref{eq:Deltacdn}) which when satisfied along with second order sufficient conditions for the local subsystem provide guarantees of its convergence over time. We then utilize the structure of the network to derive a set of conditions that guarantee convergence of a \emph{line network} in a sequential fashion starting from the root node and propagating the convergence down the line in subsequent iteration steps. Our analysis results in a relationship among power flow variables (which is trivially satisfied for a well-designed power distribution system) under which the proposed distributed optimization approach shows linear convergence. Finally, we validate the efficacy of the proposed approach by solving multiple distribution-level OPF problems. We also use simulations to provide additional insights into the convergence properties.

Our results present \emph{sufficient} conditions that guarantee convergence when satisfied.
  But we note that our conditions are somewhat different in structure from typical convergence guarantees for optimization algorithms.
  Our conditions are specified on the values of variables in the problem over one or more iterations, rather than on ranges of values of problem parameters.
  At the same time, these conditions are satisfied by the typical values taken by the system variables in power systems.
  Furthermore, our computational experiments (see Section \ref{sec:nmrclstud}) confirm the expected convergence behavior.





\section{Modeling and Problem Formulation}
In this paper $(\cdot)^{(t)}$ represents the variable at iteration step $t$,
$\underline{(.)}$ and $\overline{(.)}$ denote the minimum and maximum limit of any quantity, respectively,
and $\imi = \sqrt{-1}$.
We assume a radial single-phase power distribution network, where $\N$ and $\E$ denote the set of nodes and edges of the system.
Here, edge $ij \in \E$ identifies the distribution lines connecting the ordered pair of buses $(i,j)$ and is weighted with the series impedance of the line, represented by $r_{ij} + \imi x_{ij}$.
The set of load buses and DER buses are denoted by $\N_L$ and $\N_{D}$, respectively.

\vspace{-0.2cm}
\subsection{Network and DER Model}
Let node $i$ be the unique parent node and node $k$ is the children node for the controllable node $j$, i.e., ${k:j \rightarrow k}$ where, $\{jk\} \in \E$.
We denote $v_j$ and $l_{ij}$ as the squared magnitude of voltage and current flow at node $j$ and in branch $\{ij\}$, respectively.
The network is modeled using the nonlinear branch flow equations \cite{baran1989optimal2} shown in Eqn.~\eqref{eqModel_nonlin}.
Here, $p_{L_j}+ \imi q_{L_j}$ is the load connected at node $j \in \N_L$, $P_{ij}, Q_{ij} \in \mathbb{R}$ are the sending-end active and reactive power flows for the edge $ij$, and $p_{Dj} + \imi q_{Dj}$ is the power output of the DER connected at node $j \in \N_{D}$.

\vspace*{-0.15in}
\begin{IEEEeqnarray}{C C}
\IEEEyesnumber\label{eqModel_nonlin} \IEEEyessubnumber*
P_{ij}-r_{ij}l_{ij}-p_{L_j}+p_{Dj}= \sum_{k:j \rightarrow k} P_{jk}   \label{eqModel_nonlin1}\\
Q_{ij}-x_{ij}l_{ij}-q_{L_j}+q_{Dj}= \sum_{k:j \rightarrow k} Q_{jk} \label{eqModel_nonlin2}\\
v_j=v_i-2(r_{ij}P_{ij}+x_{ij}Q_{ij})+(r_{ij}^2+x_{ij}^2)l_{ij}\label{eqModel_nonlin3}\\
v_il_{ij} = P_{ij}^2+Q_{ij}^2 \label{eqModel_nonlin4}
\end{IEEEeqnarray}

The DERs are modeled as Photovoltaic modules (PVs) interfaced using smart inverters, capable of two-quadrant operation. Both the reactive and the active power output can be optimally controlled. For reactive power control,
at the controllable node $j \in \N_D$, the real power generation by the DER, $p_{Dj}$ is assumed to be known (measured), and the controllable reactive power generation, $q_{Dj}$, is modeled as the decision variable.
With the rating of the DER connected at node $j \in \N_D$ denoted as $S_{DRj}$, the limits on $q_{Dj}$ are given by Eqn.~\eqref{eq:DG_lim_q}. Similarly for the active power control ($p_{Dj}$), equation \eqref{eq:DG_lim_p} can be modeled as the DER assuming reactive power output to be zero; further, for both reactive and active power control, equation \eqref{eq:DG_lim_pq} may be used to model such DERs. 
%

\begin{IEEEeqnarray}{C}
\IEEEyesnumber\label{eq:DG_lim} \IEEEyessubnumber*
-\sqrt{S_{DRj}^2-p_{Dj}^2} \leq q_{Dj} \leq \sqrt{S_{DRj}^2-p_{Dj}^2} \label{eq:DG_lim_q}\\
0\leq p_{Dj} \leq S_{DR{j}} \label{eq:DG_lim_p}\\
p_{Dj}^2+q_{Dj}^2\leq S_{DRj}^2\label{eq:DG_lim_pq}
\end{IEEEeqnarray}


\subsection{Problem Formulation and Algorithm}

We recently~\cite{sadnan2020real} developed a real-time distributed controller to solve OPF problems by controlling the reactive power outputs of the DERs for distribution systems, and the method is a variant of nodal-level extension of previously developed distributed OPF \cite{sadnan2021distributed}.
In this section, we develop a similar approach for distributed OPF termed \emph{ENDiCo-OPF} by decomposing the overall problem at each node that is solved using the decomposition method developed previously in~\cite{sadnan2021distributed}. All controllable nodes in the system receive updated computed voltage and power flow quantities from their parent and children nodes, respectively, and in parallel calculate their optimum dispatches.
Briefly, each node $j \in \N_D$ solves a small-scale OPF problem defined by problem (\textbf{P1}) in Eqn.~\eqref{eq:P1}.
The reactive power dispatch $q_{Dj}$ is controlled to minimize some cost/objective function $\mathbf{f}$.
Note that the resulting nonlinear optimization problem is in five variables, $\{P_{ij}, Q_{ij}, v_{j}, l_{ij}, q_{Dj}\}$, with four equality and five inequality constraints.
Some examples of the cost function $\mathbf{f}$ include active power loss ($\mathbf {f} = r_{ij}l_{ij}$) and voltage deviation ($\mathbf {f} = (v_j-v_{\text{ref}})^2$), among others.
The ENDiCO-OPF assumes the parent node voltage and the power flow to the children node to be constant and solves the problem (\textbf{P1}) locally for the reduced network.
In this paper, we denote the sub-problem \textbf{(P1)} by the function \textbf{$F$}, i.e., the argmin of this function corresponds to the optimal solution of the OPF problem \textbf{(P1)}.
Steps of ENDiCo-OPF are presented in Algorithm \ref{alg:ENDiCo-OPF}.
We work under the following standard assumption.

\begin{asmn} \label{asm:nodeagent}
All the nodes in the network have an agent that can measure its local power flow quantities (node voltages and line flows) and communicate with neighboring nodes.
\end{asmn}

We assume the small nonlinear optimization problem (P1) in \cref{eq:qDjStep3} is solved by a commercial solver, as we did in our earlier work~\cite{sadnan2020real}.
 This assumption does not affect the convergence analysis as it is just a small subproblem.
 Alternatively, we could present a subroutine for this subproblem, based on the Augmented Lagrangian Multiplier (ALM) method \cite{Be1999} for instance, which could make it easier to include the details of this step in the overall convergence analysis.
\clearpage
 
\vspace*{-0.5in}
\begin{IEEEeqnarray}{C}
\IEEEyesnumber\label{eq:P1} \IEEEyessubnumber*
\nonumber \text{\textbf{(P1)}}\hspace*{2.5in} \\ 
\min \hspace*{0.25in} \mathbf {f}^{(t)} \hspace*{2.5in} \label{eq:P1_obj}\\
\text{s.t.}\hspace{0.25in}~P_{ij}^{(t)}-r_{ij}l_{ij}^{(t)}-p_{L_j}^{(t)}+p_{Dj}^{(t)}= \sum_{k:j \rightarrow k} P_{jk}^{(t-1)}  \hspace*{0.2in} \label{eq:P1_rl}\\
\hspace*{0.2in} Q_{ij}^{(t)}-x_{ij}l_{ij}^{(t)}-q_{L_j}^{(t)}+q_{Dj}^{(t)}= \sum_{k:j \rightarrow k} Q_{jk}^{(t-1)} \label{eq:P1_qDmxl}\\
v_j^{(t)}=v_i^{(t-1)}-2(r_{ij}P_{ij}^{(t)}+x_{ij}Q_{ij}^{(t)})+(r_{ij}^2+x_{ij}^2)l_{ij}^{(t)}\hspace{0.4cm}\label{eq:P1_vjt}\\
l_{ij}^{(t)} = \frac{(P_{ij}^{(t)})^2+(Q_{ij}^{(t)})^2}{v_i^{(t-1)}} \label{eq:P1_lijt}\\
\underline{V}^2 ~\leq v_j^{(t)} ~\leq \overline{V}^2\label{eq:P1_vjtbds} \\
-\sqrt{S_{DR_j}^2 - p^2_{D_j}(t)} ~\leq q_{D_j}^{(t)} ~\leq \sqrt{S_{DR_j}^2 - p^2_{D_j}(t)} \hspace*{0.2in} \label{eq:P1_qDjtbds}\\
l_{ij}^{(t)} \leq \left(I^{\text{rated}}_{ij}\right)^2 \label{eq:P1_lijtub}
\end{IEEEeqnarray}
%
Here, $\underline{V} = 0.95$ and $\overline{V} = 1.05$ pu are the limits on bus voltages, and $I^{\text{rated}}_{ij}$ is the thermal limit for the branch $\{ij\}$.

\begin{algorithm}[ht!]
\small
\caption{\small Equivalence of Network-based Distributed Controller for OPF (ENDiCo-OPF)}\label{alg:ENDiCo-OPF}
\SetAlgoLined
\SetKwInOut{Rx}{Receive}
\SetKwInOut{Tx}{Transmit}
\SetKwInOut{St}{Calculate}
\SetKwInOut{ND}{Node}
\SetKwInOut{TE}{iteration step }
\SetKwInOut{Stp}{ENDiCo-OPF Steps }
\ND {$\forall  j \in \N_D$}
\TE{$t$}
\St{$q^*_{Dj}$}
\Rx{$v_i^{(t-1)}$ and $P^{(t-1)}_{jk_i}+ \imi Q^{(t-1)}_{jk_i}$}
\Tx{$v_j^{(t)}$ and $P^{(t)}_{ij} + \imi Q^{(t)}_{ij}$ }
\vspace{-3pt}
\hrulefill \\
\Stp{}
\vspace{2pt}
\setcounter{AlgoLine}{0}
\ShowLn
Calculate $\sum P_{jk}^{(t-1)} + \imi Q_{jk}^{(t-1)}$ from all the $P^{(t-1)}_{jk_i} + \imi Q^{(t-1)}_{jk_i}$, received from child nodes $k_i \in \N_{jk}$ \\
\vspace{2pt}
\ShowLn
Approximate the upstream and downstream network of line $(i,j)$ with fixed value of $v_i^{(t-1)}$ and $\sum P_{jk}^{(t-1)} + \imi Q_{jk}^{(t-1)}$ \\
\vspace{2pt}
\ShowLn
Solve optimization problem (P1) for iteration step (t), i.e.,  \label{st:optP1}
\begin{equation} \label{eq:qDjStep3}
   q_{Dj}^* =  \underset{q_{Dj}}{\arg\min} \hspace{0.2cm} \mathbf {F}^{(t)}(q_{Dj})
\end{equation}
Set Reactive power output $q^{(t)}_{Dj} = q_{Dj}^*$\\
\vspace{2pt}
\ShowLn
Calculate the node voltage $v_j^{(t)}$ at node $j$ and complex power flow $P^{(t)}_{ij} + \imi Q^{(t)}_{ij}$ in the line $(i,j)$\\
\vspace{2pt}
\ShowLn
Sends $v_j^{(t)}$ and $P^{(t)}_{ij} + \imi Q^{(t)}_{ij}$ to child node $k$ and parent node $i$, respectively\\
\vspace{2pt}
\ShowLn
Receives $v_i^{(t)}$ and $P^{(t)}_{jk_i} + \imi Q^{(t)}_{jk_i}$ from parent and child nodes, respectively\\
\vspace{2pt}
\ShowLn
Move forward to the next iteration step $(t+1)$
\vspace{4pt}

	\label{algo}
\end{algorithm}

For presenting the main results on convergence of ENDiCo-OPF, we first write the optimization model (P1) in Eqn.~\eqref{eq:P1} in standard form.
We set the local variables of the bus $j$ at iteration step $t$ as
\begin{equation} \label{eq:zj}
  \bz =  \begin{bmatrix} P_{ij}^{(t)} & Q_{ij}^{(t)} & v_{j}^{(t)} & l_{ij}^{(t)} & q_{D_j}^{(t)} \end{bmatrix}^T.
\end{equation}
We also write the set of four equality constraints in Eqns.~\ref{eq:P1_rl}--\ref{eq:P1_lijt} as $\A_p(\bz)=0$ for $p=1$--$4$ and the five inequality constraints in Eqns.~\ref{eq:P1_vjtbds}--\ref{eq:P1_lijtub} as $\B_{r}(\bz)\leq 0$ for $r=1$--$5$.
Finally, setting $R_{ij} = [0~0~0~r_{ij}~0]^T$ we rewrite the original system as follows.
\begin{equation} \label{eq:P1'} \tag{P1'} 
  \begin{array}{llcll}
    \min & f(\bz)  = R_{ij}^T \bz \\
    \text{s.t.} & \A_p(\bz) & = & 0 & p=1,\dots,4 \\
                & \B_r(\bz) & \leq & 0 & r=1,\dots,5
  \end{array}
\end{equation}

Please note that only reactive power output of the DERs, $q_{D_j}$, are assumed to be the controllable variables in the OPF formulation for the convergence analysis.
However, the analysis is extendable for the active power controls, $p_{D_j}$, as well.

\section{Convergence Analysis} \label{sec:cnvgcanal}

We present theoretical guarantees for the convergence of the ENDiCo-OPF algorithm under certain standard assumptions on the network topology and structural properties of the distribution system.
We use the method of Lagrangian multipliers from nonlinear optimization coupled with techniques from linear algebra to derive the convergence property of ENDiCo-OPF.
We (i) first study the convergence of the simplest structure of a distributed network for a single iteration step (\cref{ssec:localcvgnc}).
As a first step toward generalizing this property to multiple iteration steps as well as to general networks, we present a slight modification of the ENDiCo-OPF algorithm using a new convergence parameter $\D$ (\cref{ssec:subroutine}).
We (ii) derive conditions that guarantee the convergence of the $\D$-ENDiCo-OPF algorithm for a local system over iteration steps (\cref{ssec:lclcvgnctime}).
Then we (iii) generalize the problem to achieve a network-level convergence guarantee for line networks (\cref{ssec:globalcvgnc}).
We introduce auxiliary variables to measure the difference of variable values between the adjacent iteration steps in the process of deriving guarantees for global convergence of ENDiCo-OPF over time for line networks.

\subsection{Local System Convergence Guarantees}  \label{ssec:localcvgnc}

We first study the convergence of the ENDiCo-OPF algorithm at a local level on a \emph{subsystem} under Assumption \ref{asm:nodeagent}, which refers to a system with only one communication layer that contains a middle node to receive voltages from a parent node as well as power flows from its children nodes (\cref{fig:single}).
The system consists of nodes $\{i,j,k\}$ when there is a single child node, or nodes $\{i,j,k_1,\dots,k_l\}$ in general when there are $l$ child nodes.
We assume that the system conditions are changing at a slower rate than the decision variables.
\begin{figure}[ht!]
  \centering
  \begin{subfigure}{0.99\columnwidth}
    \centering
    \includegraphics[width=0.85\columnwidth]{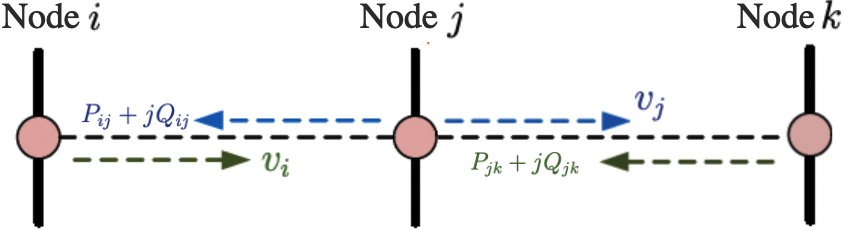}
    \caption{3-node subsystem with single child node} \label{fig:3ndW1child}
    \medskip
  \end{subfigure}
  \begin{subfigure}{0.99\columnwidth}
    \centering
    \includegraphics[width=0.85\columnwidth]{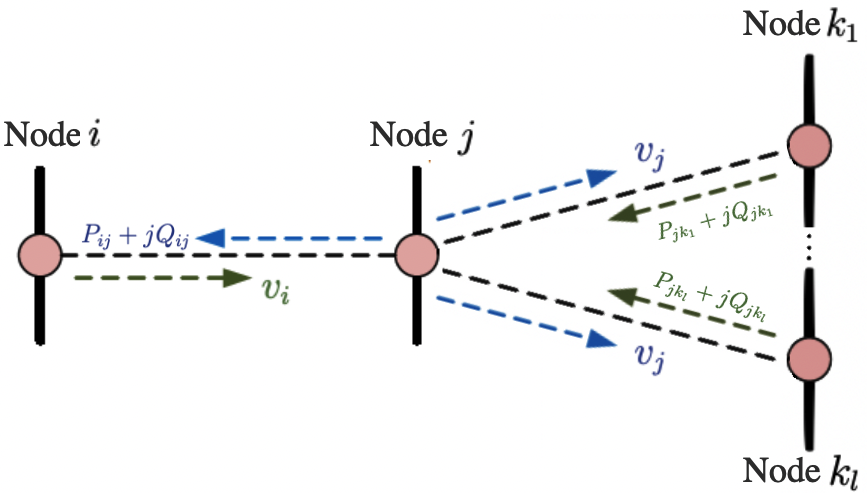}
    \caption{3-node subsystem with multiple child nodes} \label{fig:3ndWmltchild}
  \end{subfigure}
  \vspace*{-0.1in}
  \caption{Two instances of single subsystem networks}   \label{fig:single}
  \vspace*{-0.1in}
\end{figure}

We introduce auxiliary variables $\varphi_r$ to convert the inequality constraints in the system \eqref{eq:P1'} to equations.
\begin{equation} \label{eq:P2} \tag{P2} 
  \begin{array}{llcll}
    \min & f(\bz)  = R_{ij}^T \bz \\
    \text{s.t.} & \A_p(\bz) & = & 0 & p=1,\dots,4 \\
                & \B_r(\bz) + \varphi_r^2 & = & 0 & r=1,\dots,5
  \end{array}
\end{equation}

We apply second order sufficiency conditions \cite[\S3.2]{Be1999} that guarantee when $\bz^*$ is a strict local minimum of the objective function $f$ in the system (\ref{eq:P2}) under standard assumptions.

\begin{prop}[\cite{Be1999}, Proposition 3.2.1] \label{prop:sufcdnopt}
  Assume $f$ and $\A_p, \B_r$ in the System (\ref{eq:P2}) are twice continuously differentiable for all $p,r$ and let $\bz^* \in \mathbb{R}^n, \bla^*\in \mathbb{R}^p$, and $\bmu^*\in \mathbb{R}^r$  satisfy
  \begin{align*}
    \nabla_{\bz} L(\bz^*,\bla^*,\bmu^*)=0,~
    \nabla_{\bla} L(\bz^*, \bla^*, \bmu^*)=0, \hspace*{0.5in}\\
    \nabla_{\bmu} L(\bz^*, \bla^*, \bmu^*)=0, \text{ and }  \hspace*{1.5in} \\
    a^T\nabla_{\bz\bz}L(\bz^*,\bla^*,\bmu^*)a>0
    \text{ for all } a\neq 0 \text{ with } \hspace*{0.42in} \\
    \begin{bmatrix} \nabla\A(\bz^*) & \nabla\B(\bz^*) \end{bmatrix}^T a=0.
\hspace*{1in}
  \end{align*}
  Then $\bz^*$ is a strict local minimum of $f$.
\end{prop}

We now present the theorem that specifies a basic condition under which the ENDiCo-OPF algorithm is guaranteed to converge for a single iteration step on single subsystem networks as shown in \cref{fig:single}.
We show that this condition always holds for power systems operating under standard settings.
To specify the result, we consider the augmented Lagrangian function for the System (\ref{eq:P2}) for penalty parameter $c > 0$:
\begin{equation} \label{eq:augLggnP2}
\small
  \begin{aligned}
    \hspace*{-0.1in} L_c(\bz, \bphi, \bla, \bmu) & = f(\bz)+ \bla \A(\bz) + \frac{c}{2}\|\A(\bz)\|^2 ~+ \\
    & ~~~~\bmu(\B(\bz)+\bphi^2)+\frac{c}{2}\|\B(\bz)+\bphi^2\|^2 \,.
  \end{aligned}
\end{equation}

We first show that Lagrangian multipliers in Proposition \ref{prop:sufcdnopt} do exist.
Note that \eqref{eq:P2} is a convex optimization problem, and hence the strict local minimum $\bz^*$ in Proposition \ref{prop:sufcdnopt} will also be the global optimal solution.
We provide a proof similar in nature to the arguments presented by Andreani et al.~\cite{andreani2011sequential} 
\begin{lemma} \label{lem:lgrngexist}
  There exist Lagrangian multipliers $\bla^*, \bmu^*$ that along with the system variables satisfy the conditions in Proposition \ref{prop:sufcdnopt} for the System (\ref{eq:P2}).
\end{lemma}
\begin{proof}
  Let $\gamma>0$ be such that $\bz^*$ is the unique global solution of the following version of the problem in System (\ref{eq:P2}) where $\mathbf{B}(\bz^*,\gamma)$ is the $\gamma$-ball centered at $\bz^*$.
  \begin{equation} \label{eq:P2zg} \tag{P2-$\gamma$} 
    \begin{array}{ll}
      \min & f(\bz) + (1/2) \|\bz-\bz^*\|_2^2 \\
      \text{s.t.} & \A_p(\bz)~~~~~~~~ = 0,~ p=1,\dots,4 \\
      & \B_r(\bz) + \varphi_r^2 ~= 0,~r=1,\dots,5 \\
      & \bz \in \mathbf{B}(\bz^*,\gamma)
    \end{array}
  \end{equation}
  We consider the sequence of penalty subproblems associated with (\ref{eq:P2zg}) for $k \in \mathbb{N}$ (natural numbers) with the penalty parameter $\rho_k$ assumed to be large enough:
  \begin{equation} \label{eq:P2phi} \tag{P2-$\phi$}
    \hspace*{-0.15in}
    \begin{array}{l}
      \min\limits_{\bz \in \mathbf{B}(\bz^*,\gamma)}~~\phi(\bz) := f(\bz) + (1/2) \|\bz-\bz^*\|_2^2 + \\
      \hspace*{0.5in} \rho_k \left(\sum_{p=1}^4 \A_p(\bz) + \sum_{r=1}^5(\B_r(\bz) + \varphi_r^2) \right) 
    \end{array}
  \end{equation}
  By the continuity of $\phi(\cdot)$ and the compactness of $\mathbf{B}(\bz^*,\gamma)$, as $\rho_k\to \infty$, $\bz^k$ is well-defined for all $k\in \mathbb{N}$ such that $\bz^k$ is the global solution of (\ref{eq:P2phi}).
  By the boundedness of $\{\bz^k\}_{k\in\mathbb{N}}$, we have that
  \begin{align*}
    f(\bz^k) + (1/2)\|\bz^k-\bz^*\|^2_2\leq \phi(\bz^k)\leq \phi(\bz^*) = f(\bz^*)
  \end{align*}
  such that the limit point of $\bz^k$, denoted by $\bzb$ satisfies
  \begin{align*}
    f(\bzb) + (1/2) \|\bzb - \bz^*\|^2_2 \leq f(\bz^*)
  \end{align*}
  which can be fulfilled, due to the continuity and boundedness of $\phi(\cdot)$, only by making $\A(\bzb)=0$ and $\B(\bzb) + \boldsymbol{\varphi}^2=0$.
  Hence, with $\bzb \in  \mathbf{B}(\bz^*,\gamma)$, we get that $\bzb$ globally solves the problem  (\ref{eq:P2zg}),
  which further implies that $\bzb = \bz^*$ due to the uniqueness of the global solution.
  Hence for some large enough $k \in \mathbb{N}$, we have that $\bz^k$ globally solves the unconstrained minimization problem with objective function $\phi(\cdot)$.
  Hence we get that $\nabla \phi(\bz^k) = 0$ for large $k$, which gives
  \begin{equation}\label{eq:gradphi}
    \begin{array}{c}
      \nabla f(\bz^k) + \bz^k - \bz^* + \sum_{p=1}^4\lambda^k_p \nabla\A_p(\bz^k) \hspace*{0.5in} \\
      + \sum_{r=1}^5\mu^k_i(\nabla\B_r(\bz^k)+\varphi_r^2)=0
    \end{array}
  \end{equation}
  with the the sequences of Lagrangian multipliers $\{\lambda^k_p\}_{k\in\mathbb{N}},$ and $\{\mu^k_r\}_{k\in\mathbb{N}}$.
  We argue that both $\{\bla^k\}$ and $\{\bmu^k\}$ are bounded for $k\in \mathbb{N}$.
  Assume that at least one sequence is unbounded, and let $\mathcal{L}^k$ represent the maximum norm of $\bla^k_p, p=1,\dots,4$ and $\bmu^k_r, r= 1,\dots, 5$,
  Dividing (\ref{eq:gradphi}) by $\mathcal{L}^k$  and taking the limits along convergent subsequences, we get that $\mathcal{L}^k\to \infty$ and
  \begin{align}\label{eq:zk}
    \bla^* \nabla\A(\bz^*) + \bmu^*(\nabla\B(\bz^*)+\boldsymbol{\varphi}^{*2})=0
  \end{align}
  where $\|\bla^*\|, \|\bmu^*\|\geq 0$.
  However, since $\nabla \A$ and $\nabla\B + \boldsymbol{\varphi}^{2}$ are all linearly independent at $\bz^*$, we get a contradiction.
  Hence we get the boundedness of the two sequences. 
\end{proof}    

\begin{thm} \label{thm:subsyscvgce}
  Assume $f$ and $\A_p, \B_r$ in the System (\ref{eq:P2}) are twice continuously differentiable for all $p,r$ and there exists a threshold penalty parameter $\bc$ in the augmented Lagrangian (\cref{eq:augLggnP2}).
  If
  \begin{equation} \label{eq:maincdn}
    v_i^{(t-1)}-4\p^{(t)}\ri^{(t)}-4\q^{(t)}\x^{(t)}>0,
  \end{equation}
  then the ENDiCo-OPF algorithm converges for a single iteration step $t$ for all $c > \bc$. 
\end{thm}
Before presenting the proof, we note that the main condition in \cref{eq:maincdn} holds for typical values that the variables take in power systems.
At the same time, this is a sufficient condition---local convergence could happen even when \cref{eq:maincdn} does not hold. 

\begin{proof}
  We obtain this result from the application of Proposition \ref{prop:sufcdnopt}.
  We first consider the gradient of the Lagrangian function in \cref{eq:augLggnP2} with respect to $\bz$.
  We use the equations specifying the constraints in the System (\ref{eq:P2}) to simplify the expression for the gradient,
  and for the sake of brevity, we suppress the subscript of $\bz$ of $\nabla$ for the terms on the right-hand side.
  \begin{equation} \label{eq:gradLggn}
  \small
    \hspace*{-0.1in}
    \begin{aligned}
      \nabla_{\bz} L_c(\bz, \bphi, \bla, \bmu) = \nabla f(\bz) + \nabla\A(\bz)(\bla+c\A(\bz))\\
       + \nabla\B(\bz)(\bmu+c(\B(\bz)+\bphi^2))~~\\
       =\nabla f(\bz)+\bla\nabla\A(\bz)+\bmu\nabla\B(\bz) 
    \end{aligned}
  \end{equation}
  Lemma \ref{lem:lgrngexist} ensures that there exist Lagrangian multipliers $\bla^*, \bmu^*$ and variables $\bz^*, \bphi^*$ that satisfy $\nabla_{\bz} L_c(\bz^*, \bphi^*, \bla^*, \bmu^*) = 0$.

  The main technical work is involved in specifying the structure of the Hessian such that Proposition \ref{prop:sufcdnopt} will hold.
  To this end, we first note that the last sufficient condition specified in  Proposition \ref{prop:sufcdnopt} can be shown to hold by equivalently showing that the Hessian $\nabla^2_{\bz\bz}L_c$ is positive definite.
  We assume the iteration step $t$ is fixed, use the expression for $\bz$ in \cref{eq:zj}, suppress the superscripts $(t)$ of variables, let $v_i$ represent $v_i^{(t-1)}$, and let $d_j=\lambda_4^*$ to get the Hessian as
  \begin{align} 
    \nabla^2_{\bz\bz} L_c(\bz, \bphi, \bla, \bmu) = d_{j}
    \begin{bmatrix}
      2 & 0 & 0 & 0 & 0 \\
      0 & 2 & 0 & 0 & 0 \\
      0 & 0 & 0 & 0 & 0 \\
      0 & 0 & 0 & 0 & 0 \\
      0 & 0 & 0 & 0 & 0 \\
    \end{bmatrix}
    + cM \label{eq:HsnwM}
  \end{align} 
  where $M = $
  \begin{small}
  \[
    \hspace*{-0.08in}
    \begin{array}{l}
      \left[
        \hspace*{-0.07in}
        \begin{array}{cc}
          4\ri^2+4\p^2+1 & 4\ri\x+4\p\q  \\
          4\ri\x+4\p\q & 4\x^2+4\q^2+1  \\
          2\ri & 2\x \\
          -\ri \hspace*{-0.03in}-2\ri(\ri^2+\x^2)\hspace*{-0.03in}-\hspace*{-0.03in}2\p\vi & \hspace*{-0.03in}-\x\hspace*{-0.03in}-\hspace*{-0.03in}2\x(\ri^2+\x^2)\hspace*{-0.03in}-\hspace*{-0.03in}2\q\vi \\
          0 & 1
        \end{array}
        \right.\\
        \vspace*{-0.02in} \\
        \left.
        \hspace*{0.4in}
        \begin{array}{ccc}
          2\ri & -\ri-2\ri(\ri^2+\x^2)-2\p\vi & 0 \\
          2\x & -\x-2\x(\ri^2+\x^2)-2\q\vi & 1\\
          1 & -(\ri^2+\x^2) & 0 \\
          -(\ri^2+\x^2) &\ri^2+\x^2+(\ri^2+\x^2)^2+\vi^2& -x \\
          0 & -\x & 1 \\
        \end{array}
        \right].    
    \end{array}
    \]
  \end{small}
  \normalsize
  We simplify the expression in \cref{eq:HsnwM} to obtain the following expression for the Hessian, which we investigate in detail.
  \begin{align} 
    \nabla^2_{\bz\bz} L_c(\bz, \bphi, \bla, \bmu) = cM'\, \label{eq:Hsnzj}
  \end{align} 
  \noindent where $M' = $
  \begin{small}
  \[
    \begin{array}{l}
    \hspace*{-0.08in}
      \left[
        \hspace*{-0.05in}
        \begin{array}{cc}
          4\ri^2+4\p^2+1+(2d_{j}/c) & 4\ri\x+4\p\q \\
          4\ri\x+4\p\q & 4\x^2+4\q^2+1+(2d_{j}/c) \\
          2\ri & 2\x\\
          -\ri \hspace*{-0.03in}-\hspace*{-0.03in} 2\ri(\ri^2+\x^2) \hspace*{-0.03in}-\hspace*{-0.03in} 2\p\vi & \hspace*{-0.03in}-\hspace*{-0.03in}\x \hspace*{-0.03in}-\hspace*{-0.03in} 2\x(\ri^2+\x^2) \hspace*{-0.03in}-\hspace*{-0.03in} 2\q\vi \\
          0 & 1 \\
        \end{array}
        \right.\\
        \nonumber \\ 
        \left. 
        \hspace*{0.3in}
        \begin{array}{ccc}
          2\ri & -\ri-2\ri(\ri^2+\x^2)-2\p\vi & 0 \\
          2\x & -\x-2\x(\ri^2+\x^2)-2\q\vi & 1\\
          1 & -(\ri^2+\x^2) & 0 \\
          -(\ri^2+\x^2) & \ri^2+\x^2+(\ri^2+\x^2)^2+\vi^2& -x \\
          0 & -\x & 1 \\
        \end{array}
        \right].     \hspace*{0.1in}
    \end{array}
    \]
  \end{small}
  \normalsize

  Given that $c > 0$ and $\bla > \bzero$,
  we can show that the Hessian in \cref{eq:Hsnzj} is positive definite if $M'$ is so.
  And we can show that the $5 \times 5$ matrix $M'$ is positive definite by checking that all its five upper left subdeterminants are positive  \cite[Chapter~6.5]{St2016}.

  \begin{enumerate}
    \item $1 \times 1$ upper-left subdeterminant: 
      We get that
      \vspace*{-0.05in}
      \[ 4\ri^2+4\p^2+1+(2d_{j}/c) > 0 \]

      \vspace*{-0.05in}
      \noindent since $d_{j} = \lambda_4^* > 0$ and $c > 0$, and from the observation that $\ri^2, \p^2 \geq 0$.

      \smallskip
    \item $2 \times 2$ upper-left subdeterminant:

      \vspace*{-0.1in}
      \begin{small}
      \[
      \hspace*{-0.18in}
      \begin{array}{l}
        \left\vert
        \begin{matrix}
          4\ri^2+4\p^2+1+(2d_{j}/c) & 4\ri\x+4\p\q\\
          4\ri\x+4\p\q & 4\x^2+4\q^2+1+(2d_{j}/c)
        \end{matrix} \right\vert \\
        \vspace*{-0.05in} \\
        = 16(\ri^2\q^2+\p^2\x^2)+4(\ri^2+\p^2+\x^2+\q^2) + \\
         ~~\,(1/c)(8\x^2d_{j}+8\ri^2d_{j}+8\p^2d_{j}+8\q^2d_{j}+4d_{j}+(4d_{j}^2/c))\\
        ~~ -32\ri\q\p\x\\
        \vspace*{-0.1in}\\
         \geq 4(\ri^2+\p^2+\x^2+\q^2)+ \\
         ~\, (1/c)\Big(8\x^2d_{j}+8\ri^2d_{j}+8\p^2d_{j}+8\q^2d_{j}+4d_{j}+(4d_{j}^2/c)\Big)\\
        \vspace*{-0.15in}\\
        > 0 \\
        \vspace*{-0.15in}\\
      \end{array}
      \]
      \end{small}
      \hspace*{-0.1in} since $(4\ri\q-4\p\x)^2 \geq 0$, $d_{j} = \lambda_4^* > 0$, and $c > 0$.

      \medskip
    \item $3 \times 3$ upper-left subdeterminant:
      Expanding along the third row, for instance, gives
      \vspace*{-0.02in}
      \begin{small}
      \[
      \hspace*{-0.1in}
      \begin{array}{l}
        \left\vert
        \begin{matrix}
            4\ri^2+4\p^2+1+2d_{j}/c & 4\ri\x+4\p\q & 2\ri\\
            4\ri\x+4\p\q & 4\x^2+4\q^2+1+2d_{j}/c  & 2\x\\
            2\ri & 2\x & 1
        \end{matrix}
        \right\vert\\
        \vspace*{-0.1in}\\
        = \left[16(\q^2\ri^2+\p^2\q^2)+4(\ri^2+\p^2+\q^2) \, + \right. \\
        ~~~ \left. (1/c)\Big(8\ri^2d_{j}+8\p^2d_{j}+8\q^2d_{j}+4d_{j}+ (4d_{j}^2/c)\Big)+1\right]\\
        ~~~ -\left(16\p\q\ri\x+16\p^2\q^2\right)+ \\
        ~~~~ \Big(16\p\q\ri\x-16\q^2\ri^2-4\ri^2-(1/c)8\ri^2d_{j}\Big)\\
        \vspace*{-0.12in}\\
        = 4(\p^2+\q^2)+(1/c)\Big(8\p^2d_{j}+8\q^2d_{j}+4d_{j}+(4d_{j}^2/c)\Big
        ) \\
        ~~~ +1\\
        \vspace*{-0.12in}\\
        > 0
      \end{array}
      \]
      \end{small}
      \noindent following the same observations as before.

    \item $4 \times 4$ upper-left subdeterminant:
      Expanding along the fourth row gives

      \begin{small}
        \[
        \begin{array}{l}
          \hspace*{-0.2in}
          \left\vert
            \begin{array}{c@{\hspace*{0.08in}}c}
              4\ri^2+4\p^2+1+(2d_{j}/c) & 4\ri\x+4\p\q \\
              4\ri\x+4\p\q & 4\x^2+4\q^2+1+(2d_{j}/c) \\
              2\ri & 2\x \\
              \hspace*{-0.07in}-\hspace*{-0.03in} \ri \hspace*{-0.03in}-\hspace*{-0.03in} 2\ri(\ri^2+\x^2) \hspace*{-0.03in}-\hspace*{-0.03in} 2\p\vi &
               -\x\hspace*{-0.02in}-\hspace*{-0.02in}2\x(\ri^2+\x^2)\hspace*{-0.02in}-\hspace*{-0.02in}2\q\vi
            \end{array}
          \right. \\
          \vspace*{-0.05in}\\
          \hspace*{0.38in}
          \left.
            \begin{array}{cc}
              2\ri & -\ri-2\ri(\ri^2+\x^2)-2\p\vi\\
              2\x & -\x-2\x(\ri^2+\x^2)-2\q\vi\\
              1 & -(\ri^2+\x^2)\\
              -(\ri^2+\x^2) & \ri^2+\x^2+(\ri^2+\x^2)^2+\vi^2 
            \end{array}
          \hspace*{-0.03in}\right\vert \\
          \vspace*{0.05in}\\
          \hspace*{-0.21in}
          = \left[\ri+2\ri(\ri^2+\x^2)+2\p\vi \right]\left[4\x\p\q-4\ri\q^2 \right. \\
            \hspace*{0.9in}  -2\p\vi-r-(1/c)(2d_{j}\ri+4d_{j}\p\vi) \Big]\\
            \hspace*{-0.2in}+\left[-\x-2\x(\ri^2+\x^2)-2\q\vi \right]\left[ 4\p^2\x-4\ri\p\q \right.\\
            \hspace*{0.83in}  +2\q\vi+x+(1/c)(2d_{j}\x+4d_{j}\q\vi) \Big]\\
            \hspace*{-0.2in}+\left[\ri^2+\x^2\right]\left[-16\p\q\ri\vi+4\ri^2\q^2+4\x^2\p^2-4\ri^2\p^2 \right. \\
           \hspace*{0.55in} -4\x^2\q^2+4\x\q\vi+4\p\vi\ri+\ri^2+\x^2 \Big]\\
           \hspace*{-0.2in} -(1/c)\left[ 8d_{j}\x^2\p^2+8d_{j}\ri^2\p^2+8d_{j}\x^2\q^2-8d_{j}\x\q\vi+ \right. \\
           \hspace*{0.4in}  8d_{j}\q^2r^2-8d_{j}\p\ri\vi+(1/c)(4d_{j}^2\x^2+4d_{j}^2\ri^2) \Big]\\
           \hspace*{-0.2in} +\left[ \ri^2+\x^2+(\ri^2+\x^2)^2+\vi^2 \right] \left[ 4\p^2+4\q^2+1+ \right. \\
             \hspace*{0.8in}  (1/c)\Big( 8d_{j}\q^2+8d_{j}\p^2+4d_{j}+(4d_{j}^2/c) \Big) \Big]\\
           \vspace*{-0.1in} \\
           \hspace*{-0.25in} = 8\x\p\q\vi+4\p^2\ri^2+4\q^2\x^2+ \\
           \hspace*{-0.25in} \vi\Big(\vi-4\p\ri-4\x\q\Big)
           +(1/c) \Big[8d_{j}\p^2\ri^2+8d_{j}\q^2\ri^2 + \\
           \hspace*{-0.25in} 8d_{j}\p^2\x^2+8d_{j}\q^2\x^2+2d_{j}v^2_i+2d_{j}v_i\Big(\vi \hspace*{-0.03in}-\hspace*{-0.03in} 4\p\ri \hspace*{-0.03in}-\hspace*{-0.03in} 4\x\q\Big) \\
           \hspace*{0.22in} +2d_{j}\ri^2+2d_{j}\x^2 +(1/c)\Big(4d_{j}^2\ri^2+4d_{j}^2\x^2+4d_{j}^2\vi^2\Big) \Big]\\
           \vspace*{-0.07in} \\
           \hspace*{-0.2in} > 0
        \end{array}
        \]
      \end{small}
      following the assumption of the theorem in \cref{eq:maincdn} which guarantees that $\vi-4\p\ri-4\x\q > 0$.

      \bigskip
    \item $5 \times 5$ upper-left subdeterminant:
      Expanding along the fifth row gives

      \begin{small}
      \[
      \begin{array}{l}
        \hspace*{-0.2in}
        \left|
          \hspace*{-0.05in}
          \begin{array}{c@{\hspace*{0.12in}}c}
            4\ri^2+4\p^2+1+(2d_{j}/c) & 4\ri\x+4\p\q \\
            4\ri\x+4\p\q & 4\x^2+4\q^2+1+(2d_{j}/c) \\
            2\ri & 2\x\\
            \hspace*{-0.02in}-\hspace*{-0.02in}\ri \hspace*{-0.02in}-\hspace*{-0.02in}2\ri(\ri^2+\x^2) \hspace*{-0.02in}-\hspace*{-0.02in}2\p\vi & \hspace*{-0.05in} -\hspace*{-0.02in}\x \hspace*{-0.02in}-\hspace*{-0.02in}2\x(\ri^2+\x^2)\hspace*{-0.03in}-\hspace*{-0.02in}2\q\vi \\
            0 & 1 \\
          \end{array}
          \right.\\
          \nonumber \\ 
          \left. 
          \hspace*{0.1in}
          \begin{array}{ccc}
            2\ri & -\ri-2\ri(\ri^2+\x^2)-2\p\vi & 0 \\
            2\x & -\x-2\x(\ri^2+\x^2)-2\q\vi & 1\\
            1 & -(\ri^2+\x^2) & 0 \\
            -(\ri^2+\x^2) & \ri^2+\x^2+(\ri^2+\x^2)^2+\vi^2& -x \\
            0 & -\x & 1 \\
          \end{array}
          \hspace*{-0.05in} \right|      \\
          \vspace*{-0.04in}\\
          \hspace*{-0.17in} = -4\ri\x\p\q-4\ri^2\p^2+4\ri\p\vi+2\q\x\vi-\vi^2+ \\
          \hspace*{-0in} (1/c)(-2d_{j}\ri^2-2d_{j}\vi^2+4d_{j}\q\x\vi) -4\ri\x\p\q \\
          \hspace*{0in} -4\q^2\x^2 + 2\q\vi\x + (1/c)\left[-8d_{j}\p^2\x^2-8d_{j}\q^2\x^2  \right.\\
            \left. ~+4d_{j}\q\x\vi -2d_{j}\x^2-(1/c)4d_{j}^2\x^2 \right]+8\x\ri\p\q \\
          \hspace*{0in}\left. +4\p^2\ri^2 + 4\q^2\x^2 - 4\p\ri\vi-4\q\x\vi+\vi^2 \right. \\
            \hspace*{0in} +(1/c)\left[8d_{j}\p^2\ri^2 + 8d_{j}\q^2\ri^2 + 8d_{j}\p^2\x^2+8d_{j}\q^2\x^2\right. \\
              \hspace*{0.42in} \left. -8d_{j}\p\ri\vi - 8d_{j}\q\x\vi + 4d_{j}\vi^2 + 2d_{j}\ri^2 \right.\\
              \hspace*{0.42in} \left. + 2d_{j}\x^2 + (1/c)(4d_{j}^2\ri^2+4d_{j}^2\x^2+4d_{j}^2\vi^2)\right] \\
          \vspace*{-0.04in}\\
          \hspace*{-0.2in} = (1/c)\left[ 8d_{j}\p^2\ri^2+8d_{j}\q^2\ri^2+ 2d_{j}\vi\Big(\vi-4\p\ri\Big)+ \right. \\
          \hspace*{0.3in} \left. (1/c)(4d_{j}^2\ri^2+4d_{j}^2\vi^2) \right] \\
          \hspace*{-0.2in} > 0
      \end{array}
      \]
      \end{small}
      \noindent as the assumption in \cref{eq:maincdn} gives that $\vi-4\p\ri > 0$.
  \end{enumerate}

  Hence we get that $\bz^*$ satisfying the setting of Proposition \ref{prop:sufcdnopt} is a strict local minimum of $f$.
  Thus we get the convergence of ENDiCo-OPF algorithm for a single iteration step when $f$ is convex. 
\end{proof}
  
 \vspace*{-0.2in}
\subsection{Modification of Algorithm: \texorpdfstring{$\D$}{}-ENDiCo-OPF} \label{ssec:subroutine}

Developing similar convergence guarantees for local systems over time and further extending the same to global networks present non-trivial challenges.
Motivated by the convergence behavior in practice of the original ENDiCO algorithm \cite{sadnan2020real}, we present a modification of the algorithm where a convergence parameter $\Dt \geq 1$ is adaptively chosen in each iteration step $t$.
We then derive conditions generalizing ones in Theorem \ref{thm:subsyscvgce} that guarantee convergence of a local system over time when $\Dt=1$.

Motivated by the definition of bi-Lipschitz functions in real analysis \cite[Def 28.7]{Ye2000}, we define the following condition for the convergence of voltage $v_j$ over time as captured by the parameter $\Dt \geq 1$:
\begin{equation} \label{eq:Deltacdn}
\small
  \frac{1}{\Dt} v_j^{(t-1)} ~ \leq~ v_j^{(t)} ~\leq~ \Dt v_j^{(t-1)}.
\end{equation}
Note that $\Dt = 1$ in the above condition implies $v_j^{(t)} = v_j^{(t-1)}$, which certifies convergence of $v_j$ over time.

The only modification we make to the ENDiCo-OPF algorithm is the use of the scaling parameter $\Dt$ that is suitably initialized by the user ($\D^{(0)} > 1$) and the addition of Step 8 presented in Algorithm \ref{alg:DENDiCo}.
Subsequently, the final step advancing the algorithm to the next iteration step is now numbered as Step 9.
\begin{algorithm}
\small
  \SetKwInOut{Input}{input}
  \Input{Set $\D^{(0)} > 1$}
  \vspace*{-0.1in}
  \hrulefill \\
  \setcounter{AlgoLine}{7}
  \ShowLn
  \uIf{$\vt{j}=\vm{j}$}
      {Stop}
  \uElseIf{Eqn.~(\ref{eq:Deltacdn}) holds}
      { $\D^{(t+1)} = 1+\frac{\Dt-1}{2}$}
  \uElse{$\D^{(t+1)} = 2\Dt-1$}
  \setcounter{AlgoLine}{8}
  \ShowLn
  Move forward to the next iteration step $(t+1)$    
  \caption{$\D$-ENDiCo-OPF: Changes to Algo. \ref{alg:ENDiCo-OPF}} \label{alg:DENDiCo}
\end{algorithm}

\vspace*{-0.2in}
\subsection{Local Convergence over Time} \label{ssec:lclcvgnctime}

We now consider the convergence under $\D$-ENDiCo-OPF of a local subsystem with a single child node, i.e., one consisting of nodes $\{i,j,k\}$ as shown in \cref{fig:3ndW1child}.
The system of equations (\ref{eq:bdcase-3}) corresponds to System \ref{eq:P1} but has subsystems for iteration steps $t$ and $(t+1)$ (sub-equations (\ref{eq:bdcase-3-t}) and (\ref{eq:bdcase-3-tp1})) as well as a subsystem for determining $\pt{jk}$ and $\qt{jk}$ (\ref{eq:bdcase-3-PQ}) as well as the convergence condition (\ref{eq:Deltacdn}) in sub-equation (\ref{eq:bdcase-3-Dcdn}) apart from bounds (\ref{eq:bdcase-3-bds}).

\vspace*{-0.1in}
\begin{subequations}
  \label{eq:bdcase-3}
\small
  \begin{align}
    \min~ &\rt{ij}\lt{ij}+\rt{ij}\la{ij}+\rt{jk}\lt{jk}\\
    \text{s.t.}~~\,& \nonumber \\ 
    &\text{Constraints for iteration step $t$ :} \label{eq:bdcase-3-t}\\
    & \hspace*{-0.2in}  \nonumber 
    \begin{array}{l}
      \,-\rt{ij}\lt{ij} = P_{jk}^{(t-1)} - \pt{ij} + p_{Lj}^{(t)} - p_{Dj}^{(t)}\\
      \vspace*{-0.1in} \\
      -\xt{ij}\lt{ij} = Q_{jk}^{(t-1)} - \qt{ij} + q_{Lj}^{(t)} - q_{Dj}^{(t)}\\
      \vspace*{-0.1in} \\
      \hspace*{0.27in} \vt{j} = \vm{i} - 2(\rt{ij}\pt{ij}+\xt{ij}\qt{ij}) +  (\rt{ij}^2+\xt{ij}^2)\lt{ij}\\
      \vspace*{-0.1in} \\
      \hspace*{0.27in}\lt{ij} = \left( (\pt{ij})^2+(\qt{ij})^2 \right) / \vm{i}
    \end{array}  \\
    \vspace*{-0.2in} \nonumber \\
    &\text{Solve for $\pt{jk},\qt{jk}$ :} \label{eq:bdcase-3-PQ} \\
    & \hspace*{-0.2in} \nonumber
    \begin{array}{l}
      -\rt{jk}\lt{jk} = - \pt{jk} + p_{Lk}^{(t)} - p_{Dk}^{(t)}\\
      \vspace*{-0.1in} \\
      -\xt{jk}\lt{jk} = - \qt{jk} + q_{Lk}^{(t)} - q_{Dk}^{(t)}\\
      \vspace*{-0.1in} \\
      \hspace*{0.27in} \vt{k}= \vm{j} - 2(\rt{jk}\pt{jk}+\xt{jk}\qt{jk}) +  (\rt{jk}^2+\xt{jk}^2)\lt{jk}\\
      \vspace*{-0.1in} \\
      \hspace*{0.27in} \lt{jk} = (\pt{jk})^2+(\qt{jk})^2/\vm{j}
    \end{array}  \\
    \vspace*{-0.2in} \nonumber \\
    &\text{Constraints for iteration step $t+1$ :} \label{eq:bdcase-3-tp1}     \\
    & \hspace*{-0.3in}  \nonumber 
    \begin{array}{l}
    -\rt{ij}\la{ij} = P_{jk}^{(t)} - \pt{ij} + p_{Lj}^{(t+1)} - p_{Dj}^{(t)}\\
    \vspace*{-0.1in} \\
    -\xt{ij}\la{ij} = Q_{jk}^{(t)} - \qt{ij} + q_{Lj}^{(t+1)} - q_{Dj}^{(t+1)}\\
    \vspace*{-0.1in} \\
    \hspace*{0.15in} \va{j}= \vt{i} - 2(\rt{ij}\pa{ij}+\xt{ij}\qa{ij}) + (\rt{ij}^2+\xt{ij}^2)\la{ij}\\
    \vspace*{-0.1in} \\
    \hspace*{0.27in} \la{ij} = \left((\pa{ij})^2+(\qa{ij})^2\right)/ \vt{i}
    \end{array}  \end{align}
\end{subequations}

\vspace*{-0.1in}
\addtocounter{equation}{-1}
\begin{subequations}
\small
  \label{eq:bdcase-3_2}
  \addtocounter{equation}{4}
  \begin{align}
    &\text{Bounding $\va{j}$ in terms of $\vt{j}$ :}  \label{eq:bdcase-3-Dcdn}\\
    & \left(1/\D^{(t+1)}\right) \vt{j} ~\leq~  \va{j} ~\leq~ \D^{(t+1)} \, \vt{j} \nonumber \\
    \vspace{-0.1in} \nonumber \\
    &\text{Bounds on variables :} \label{eq:bdcase-3-bds} \\
    &~~~~~~~\underline{V}^2\leq \vt{i},\vt{j}, \vt{k}, \va{j}\leq \Bar{V}^2 \nonumber\\
    &-\sqrt{S_{DR_j}^2-p_{Dj}^2}\leq q_{Dj}^{(t)}\leq \sqrt{S_{DR_j}^2-p_{Dj}^2} \nonumber \\
    &-\sqrt{S_{DR_k}^2-p_{Dk}^2}\leq q_{Dk}^{(t)}\leq \sqrt{S_{DR_k}^2-p_{Dk}^2} \nonumber \\
    &-\sqrt{S_{DR_j}^2-p_{Dj}^2}\leq q_{Dj}^{(t+1)}\leq \sqrt{S_{DR_j}^2-p_{Dj}^2} \nonumber \\
    & \lt{ij}\leq \left(I_{ij}^{\text{rated}}\right)^2, \hspace*{0.1in} \lt{jk}\leq \left(I_{jk}^{\text{rated}}\right)^2, \hspace*{0.1in} \la{ij}\leq \left(I_{ij}^{\text{rated}}\right)^2 \nonumber
  \end{align}
\end{subequations}

\begin{thm}
  \label{thm:single-bd}
  Assume that the objective function and all constraints in System (\ref{eq:bdcase-3}) are twice differentiable.
  Then the local subsystem converges at iteration step $t$ if the $\D$-convergence condition (\ref{eq:Deltacdn}) holds with $\Dt=1$ as well as the following conditions hold.

  \vspace*{-0.1in}
  \begin{equation}  
  \small
  \label{eq:cdnsingle-bd}
    \begin{aligned}
      &\xt{ij} - \rt{ij}\geq 0  \\
      &\vm{i}\vm{j}-4\pa{ij}\rt{ij}\vm{j}-4\pt{jk}\rt{jk}\vm{i}\geq 0  \\
      &\vm{i} \vm{j} - 4\pt{ij}\rt{ij}\vm{j}-4\pt{jk}\rt{jk}\vm{i}\geq 0  \\
      &\vm{i} - 4\pa{ij}\rt{ij} - 4\pt{ij}\rt{ij}\geq 0  \\
      &\vm{i}-4\pa{ij}\rt{ij}-2\pt{ij}\rt{ij}-2\pt{ij}\xt{ij}\geq 0  \\
      &\vm{i}-4\pt{ij}\rt{ij}-4\pt{ij}\xt{ij}\geq 0  \\
      &\vm{i}-\pa{ij}\rt{ij}-\pt{jk}\rt{jk}\geq 0  \\
      &\vm{i}\vm{j} - 4\pt{jk}\rt{jk}\vm{i}   \\
      &\hspace*{0.7in} -2\pt{ij}\rt{ij}\vm{j}-2\pt{ij}\xt{ij}\vm{j}\geq 0  \\
      &({\vm{j}})^3 - 4\pa{ij}\rt{ij}\vm{i}\vm{j}  \\
      &\hspace*{0.52in} -2\pt{jk}\rt{jk}\vm{i}-4\pa{ij}\rt{ij}\geq 0   \\
      & \vm{i}-8\pt{ij}\rt{ij} \geq 0  \\
      & \vm{i}-8\qt{ij}\xt{ij} \geq 0  \\
      & \vm{i}-8\pt{jk}\rt{jk} \geq 0
    \end{aligned}
  \end{equation}
\end{thm}
We note that most conditions in \cref{eq:cdnsingle-bd} hold for typical values of the variables except the first one: $\xt{ij} \geq \rt{ij}$ may not always hold.
Similar to the condition specified in Theorem \ref{thm:subsyscvgce} for local convergence, the above system also gives a sufficient condition.
In practice, we observe convergence even when $\xt{ij} \geq \rt{ij}$ does not hold.

\begin{proof}
  Analogous to the proof of Theorem \ref{thm:subsyscvgce}, we construct the augment Lagrangian function of System \eqref{eq:bdcase-3} w.r.t.~to the following variable vector (in place of the one in \cref{eq:zj}):
  \begin{equation} \label{eq:zj3t}
    \hspace*{-0.1in}
    \begin{aligned}
      \bz =  \left[ \begin{matrix}
          \pt{ij} & \qt{ij} & \vt{j} & \lt{ij} & q_{Dj}^{(t)} & \ldots \hspace*{0.65in}
        \end{matrix}\right.\\
        \left.\begin{matrix}
          \pt{jk} & \qt{jk} & \vt{k} & \lt{jk} & q_{Dk}^{(t)} & \ldots \hspace*{0.65in}
        \end{matrix}\right.\\
        \left.\begin{matrix}
           \pa{ij} & \qa{ij} & \va{j} & \la{ij} & q_{Dj}^{(t+1)} & 
        \end{matrix} \right]^T
      \end{aligned}
  \end{equation}
  
  which gives the following first-order gradient:
  \[
  \begin{aligned}
    \hspace*{-0.04in} \nabla_{\bz} L \hspace*{-0.02in} = \hspace*{-0.05in}
    \left[
      \begin{matrix}
        -1 &  0 & -2\rt{ij} & 2\pt{ij} & 1 & 1 & 0 \\
        0 & -1 & -2\xt{ij} & 2\qt{ij} & 0 & 0 & 0 \\
        0 &  0 &    -1     &     0    & 0 & 0 & 0 \\
        \rt{ij} & \xt{ij} & \rt{ij}^2+\xt{ij}^2 & -\vt{j} & 0 & 0 & 0  \\
        0  & -1 & 0 & 0 & 0 & 0 & 0  \\
        0  & 0 & 0 & 0 & -1 & 0 & -2\rt{jk} \\
        0  & 0 & 0 & 0 & 0 & -1 & -2\xt{jk} \\
        0  & 0 & 0 & 0 & 0 & 0 & -1 \\
        0  & 0 & 0 & 0 & \rt{jk} & \xt{jk} & \rt{jk}^2+\xt{jk}^2 \\
        0  & 0 & 0 & 0 & 0 & -1 & 0  \\
        0  & 0 & 0 & 0 & 0 & 0 & 0 \\
        0  & 0 & 0 & 0 & 0 & 0 & 0 \\
        0  & 0 & 0 & 0 & 0 & 0 & 0 \\
        0  & 0 & 0 & 0 & 0 & 0 & 0 \\
        0  & 0 & 0 & 0 & 0 & 0 & 0 
      \end{matrix}        
      \right.
  \end{aligned}
  \]
  \[
  \begin{aligned}
      \hspace*{-0.05in}
      \left.
      \begin{matrix}
          0 & 0 & 0 & 0 & 0 & 0 & 0   \\
          0 & 0 & 0 & 0 & 0 & 0 & 0   \\
          0 & 0 & 0 & 0 & 0 & 0 & 0  \\
          0 & 0 & 0 & 0 & 0 & 0 & 0  \\
          0 & 0 & 0 & 0 & 0 & 0 & 0  \\
          2\pt{jk} & 0 & 0 & 0 & 0 & 0 & 0  \\
          2\qt{jk} & 0 & 0 & 0 & 0 & 0 & 0  \\
          0 & 0 & 0 & 0 & 0 & 1 & 1\\
          -\vt{k} & 0 & 0 & 0 & 0 & 0 & 0  \\
          0 & 0 & 0 & 0 & 0 & 0 & 0  \\
          0 &   -1    &    0    &    2\rt{ij}         & 2\pa{ij} &   0   &   0  \\
          0 &    0    &   -1    &   -2\xt{ij}         & 2\qa{ij} &   0   &   0  \\
          0 &    0    &    0    &        -1           &     0    &   0   &   0  \\
          0 & \rt{ij} & \xt{ij} & \rt{ij}^2+\xt{ij}^2 & -\va{j}  & -1/\Dt & -\Dt  \\
          0 &    0    &   -1    &         0           &     0    &   0   &   0  \\
        \end{matrix}        
      \,\right]
  \end{aligned}
  \]
  We then considered the positive definiteness of the Hessian $\nabla_{\bz\bz}L$ to derive conditions under which all its upper left sub-determinants are positive.
  We used symbolic computation in Maple \cite{maple} to simplify expressions for the higher order sub-determinants.
 These expressions turn out to be too tedious to present here in full, so we present them online \cite{Determinants}.
  Examining expressions of the sub-determinants led to the conditions in \cref{eq:cdnsingle-bd}, which when satisfied guarantee positive definiteness of the Hessian.
  Our approach was similar to the one we used in the proof of Theorem \ref{thm:subsyscvgce} where we paired or grouped positive terms in the determinant expressions with negative ones so that their sum is guaranteed to be positive (when the condition in \cref{eq:maincdn} is satisfied).

  While $\Dt$ appears in the last two columns of the gradient $\nabla_{\bz}L$, it does not appear in the Hessian $\nabla_{\bz\bz}L$ and hence does not appear in \cref{eq:cdnsingle-bd}.
  But these conditions holding along with the $\D$-convergence condition (\ref{eq:Deltacdn}) with $\Dt=1$ guarantee the convergence over time of the local subsystem.
\end{proof}

\subsection{Global System Convergence for Line Systems} \label{ssec:globalcvgnc}

As the next generalization, we consider a \emph{line} network with $n \geq 4$ nodes.
To keep notation simple, we label the nodes $\{1, \dots, n\}$ where $1$ is the source node and $n$ is the final leaf node
(note that $n$ is also the number of nodes in the line).
Not surprisingly, this system presents even more challenges to derive conditions that guarantee global convergence over time.
Guided by the convergence behavior of the original ENDiCo algorithm in practice, we derive conditions under a suitable assumption that when satisfied in a time sequential manner guarantee global convergence of the line system in a sequential manner, i.e., starting with the first local subsystem $\{1, 2, 3\}$, moving to the next local subsystem $\{2, 3, 4\}$, and so on until the last local subsystem $\{n-2, n-1, n\}$.

\begin{asmn} \label{asm:cvgcpres}
  Given two adjacent and overlapping local subsystems $\{i, i+1, i+2\}$ and $\{i+1, i+2, i+3\}$ from the line network $\{1, \dots, n\}$ where the first subsystem is convergent at iteration step $t_i$, the convergence of the controllable node $i+1$ is preserved in the next iteration step $t_{i+1} \geq t_i+1$ for the second subsystem.
\end{asmn}

We get the guarantee of global convergence of a line system under Assumption \ref{asm:cvgcpres} by repeatedly applying Theorem \ref{thm:single-bd} in sequence going from the source node to the final leaf node.

\begin{thm} \label{thm:globalcvgnc}
  The global convergence of a line system $\{1, \dots, n\}$ is guaranteed under Assumption \ref{asm:cvgcpres} if the following conditions hold at each local subsystem $\{i, i+1, i+2\}$ in a sequential manner with iteration steps $t_{i+1} \geq t_i + 1$ for $i=1,\dots,n-2$.
  \begin{equation} \label{eq:LnDeltacdn}
    \begin{aligned}
      \left( 1/\D^{t_i} \right) v_{i+1}^{(t_i-1)} ~ \leq~ v_{i+1}^{(t_i)} ~\leq~ \D^{t_i} v_{i+1}^{(t_i-1)} \\
      \text{ holds for } \D^{t_i} = 1 \text{ for } i=1, \dots, n-2.
    \end{aligned}
  \end{equation}

  \begin{equation}  \label{eq:LnCdns} 
    \begin{aligned}
      &\xti{i,i+1} - \rti{i,i+1}\geq 0  \\
      &\vmi{i}\vmi{i+1}-4\pai{i,i+1}\rti{i,i+1}\vmi{i+1}- \\
      &\hspace*{0.87in} 4\pti{i+1,i+2}\rti{i+1,i+2}\vmi{i}\geq 0  \\
      &\vmi{i} \vmi{i+1} - 4\pti{i,i+1}\rti{i,i+1}\vmi{i+1}-\\
      &\hspace*{0.87in} 4\pti{i+1,i+2}\rti{i+1,i+2}\vmi{i}\geq 0  \\
      &\vmi{i} - 4\pai{i,i+1}\rti{i,i+1} - 4\pti{i,i+1}\rti{i,i+1}\geq 0  \\
      &\vmi{i}-4\pai{i,i+1}\rti{i,i+1}-2\pti{i,i+1}\rti{i,i+1}- \\
      &\hspace*{1.55in} 2\pti{i,i+1}\xti{i,i+1}\geq 0  \\
      &\vmi{i}-4\pti{i,i+1}\rti{i,i+1}-4\pti{i,i+1}\xti{i,i+1}\geq 0  \\
      &\vmi{i}-\pai{i,i+1}\rti{i,i+1}-\pti{i+1,i+2}\rti{i+1,i+2}\geq 0 \\
      &\vmi{i}\vmi{i+1} - 4\pti{i+1,i+2}\rti{i+1,i+2}\vmi{i} \\
      &~-2\pti{i,i+1}\rti{i,i+1}\vmi{i+1}- 2\pti{i,i+1}\xti{i,i+1}\vmi{i+1}\geq 0  \\
      &({\vmi{i+1}})^3 - 4\pai{i,i+1}\rti{i,i+1}\vmi{i}\vmi{i+1}  \\
      &~-2\pti{i+1,i+2}\rti{i+1,i+2}\vmi{i}-4\pai{i,i+1}\rti{i,i+1}\geq 0   \\
      &\vmi{i}-8\pti{i,i+1}\rti{i,i+1} \geq 0  \\
      &\vmi{i}-8\qti{i,i+1}\xti{i,i+1} \geq 0  \\
      &\vmi{i}-8\pti{i+1,i+2}\rti{i+1,i+2} \geq 0 
    \end{aligned}
  \end{equation}
\end{thm}
\vspace*{0.2in}

We note that most conditions in \cref{eq:LnCdns} hold for typical values of the variables except the first one.
But we observe global convergence in practice in the sequential manner even when $\xti{i,i+1} \geq \rti{i,i+1}$ may not hold.

\begin{proof}
  Note that conditions in \cref{eq:LnCdns} are the same as those in \cref{eq:cdnsingle-bd} from Theorem \ref{thm:single-bd} applied for the local subsystem $\{i, i+1, i+2\}$ in place of $\{i,j,k\}$.
  Since conditions in \cref{eq:LnDeltacdn,eq:LnCdns} hold for $i=1$, the first local subsystem $\{1,2,3\}$ is convergent at iteration step $t_1$ following Theorem \ref{thm:single-bd}.
  By Assumption \ref{asm:cvgcpres}, the convergence of node $2$ is preserved in the next iteration step, and hence it can be treated as the fixed source node for the next local subsystem $\{2,3,4\}$.
  The convergence of this subsystem is then guaranteed at time $t_2 \geq t_1+1$ by conditions in \cref{eq:LnDeltacdn,eq:LnCdns} holding for $i=2$.
  The overall result follows by the sequential application of Theorem \ref{thm:single-bd}.
\end{proof}

\section{Numerical Study} \label{sec:nmrclstud}
In this section, we demonstrate the convergence properties of the ENDiCo-OPF algorithm with the help of numerical simulations.
We also validate the optimality of our algorithm by comparing its results with those of a centralized solution.
These simulations not only justify the convergence analysis of the method but also showcase the efficacy of the proposed real-time distributed controller to attain optimal power flow solutions.
After attaining the optimal dispatch, the controller shares the computed boundary variables with it its neighbor instead of implementing and measuring the variables.

\begin{figure}[ht!]
  \centering
    \includegraphics[width=0.9\columnwidth]{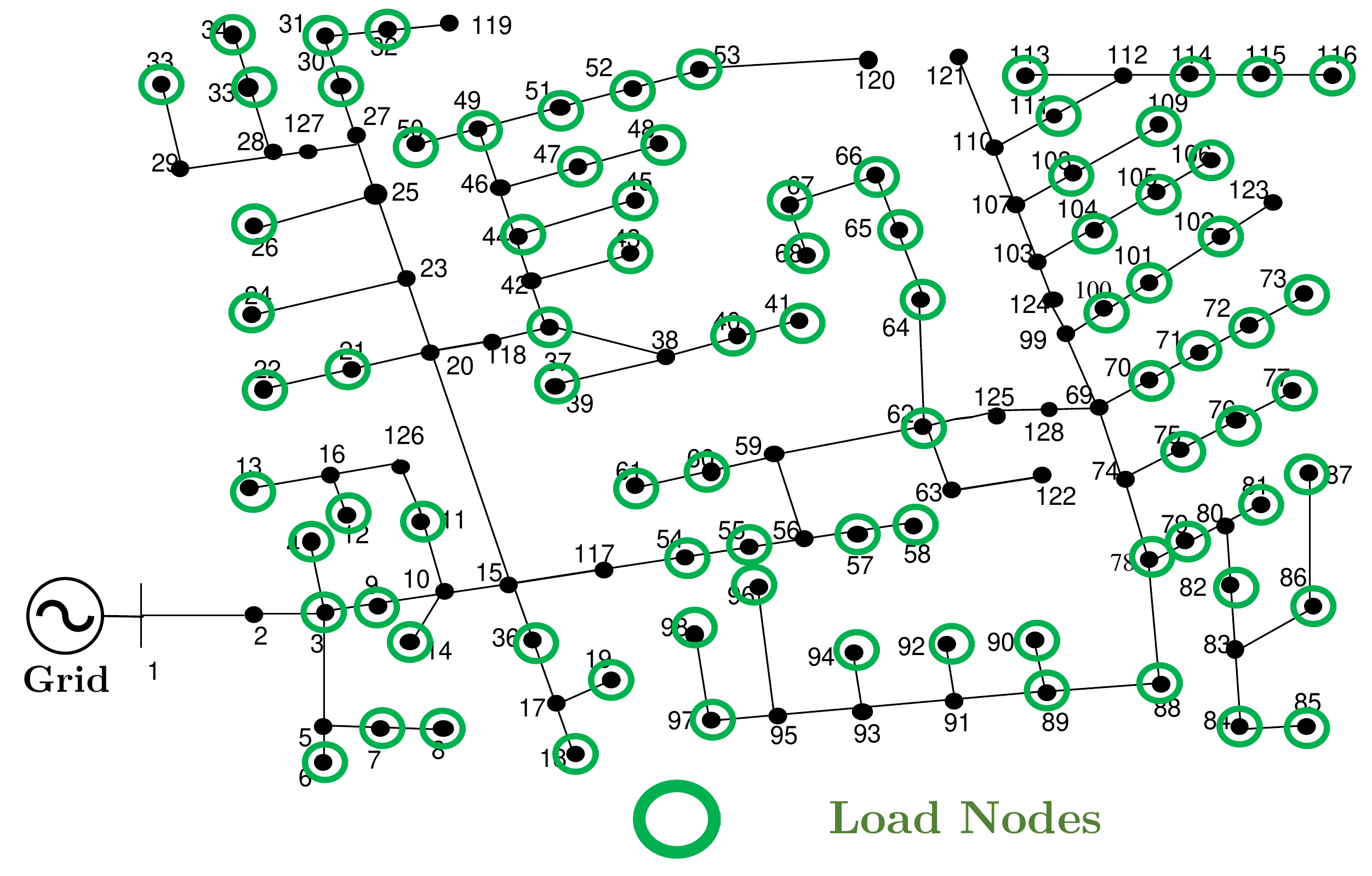}
  \caption{IEEE-123 Bus Test System
  }
  \vspace*{-0.5cm}
  \label{Test_sys}
\end{figure}

\vspace*{-0.1in}
\subsection{Simulated System and Results} 
As a test system, we simulated a balanced IEEE-123 bus system with a maximum of 85 DERs (PVs) connected to the network (Fig. \ref{Test_sys}), where the DER/PV penetration can vary from 10\% to 100\%.
As a cost function, we have simulated both (i) active power loss minimization ($f = r_{ij}l_{ij}$), and (ii) voltage deviation ($\Delta V$) minimization ($f = {(v_j-v_{\mbox{ref}})^2}$ ) optimization problems.
\subsubsection{Residual and Objective Value Convergence}
We have simulated the test system with 10\%, 50\%, and 100\% DER/PV penetration cases for both active power loss minimization and $\Delta V$ minimization.
The system converged after 42 iterations for all six cases (Fig.~\ref{convergence}).
For the loss minimization objective, the maximum border residual goes below the tolerance value of $10^{-3}$ after the 42nd iteration.
The objective values for the loss minimization OPF is 26.5 kW, 19.6 kW, and 11.8 kW, for 10, 50, and 100\% DER penetration, respectively.
Similarly, for $\Delta V$ minimization, we can see that the maximum residual goes below the tolerance after the 42nd iteration as well.
Thus the convergence is related to the network size, but not to the number of controllable variables. 
\begin{figure}[t!]
  \centering
  \begin{subfigure}{0.75\columnwidth}
    \hspace*{-0.15in}
    \includegraphics[width=\columnwidth]{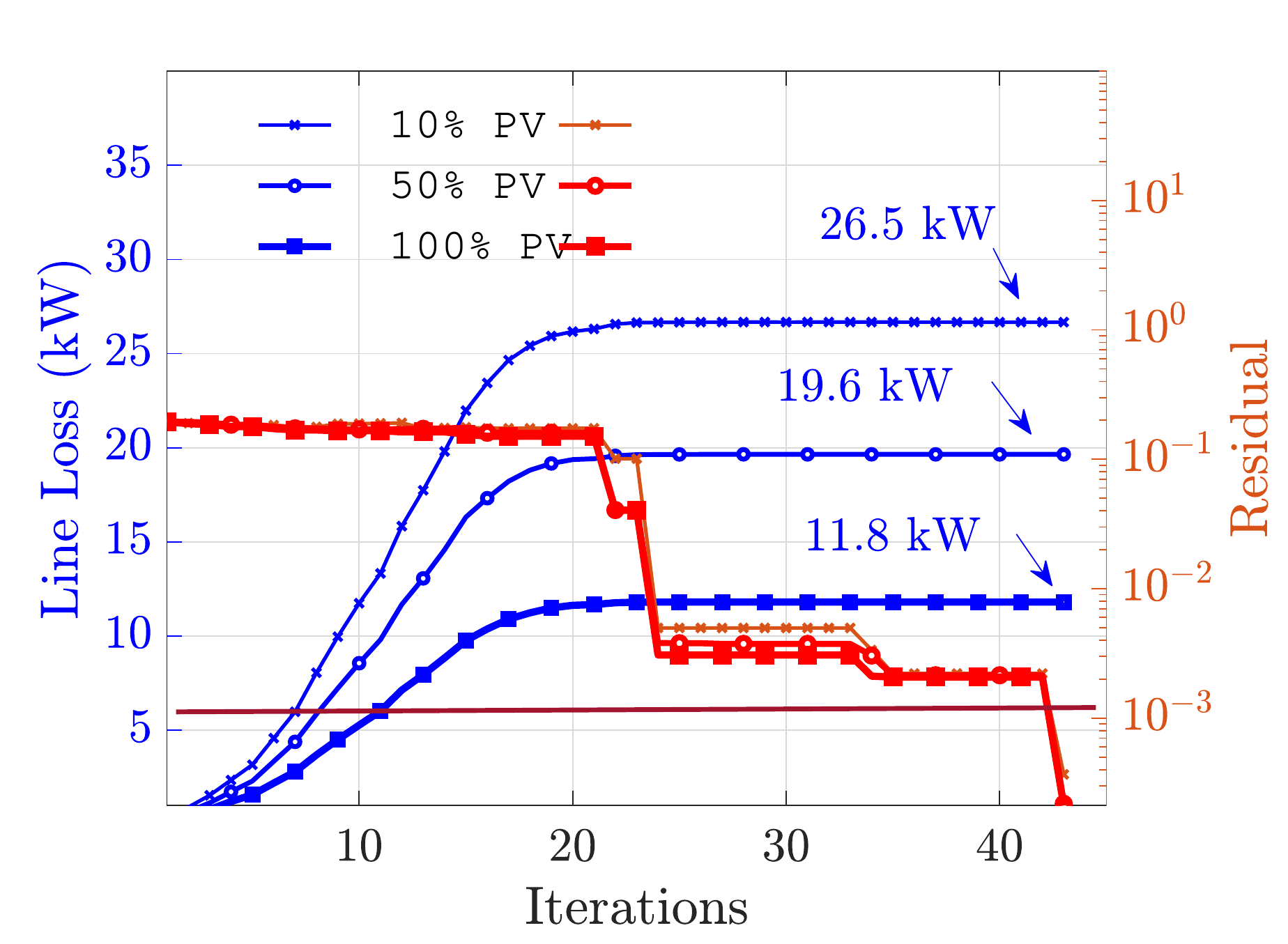}
    \vspace*{-0.4cm}
    \caption{Loss Minimization}
  \end{subfigure}
  \vspace*{-0.3cm}
  \begin{subfigure}{0.78\columnwidth}
    \includegraphics[width=\columnwidth]{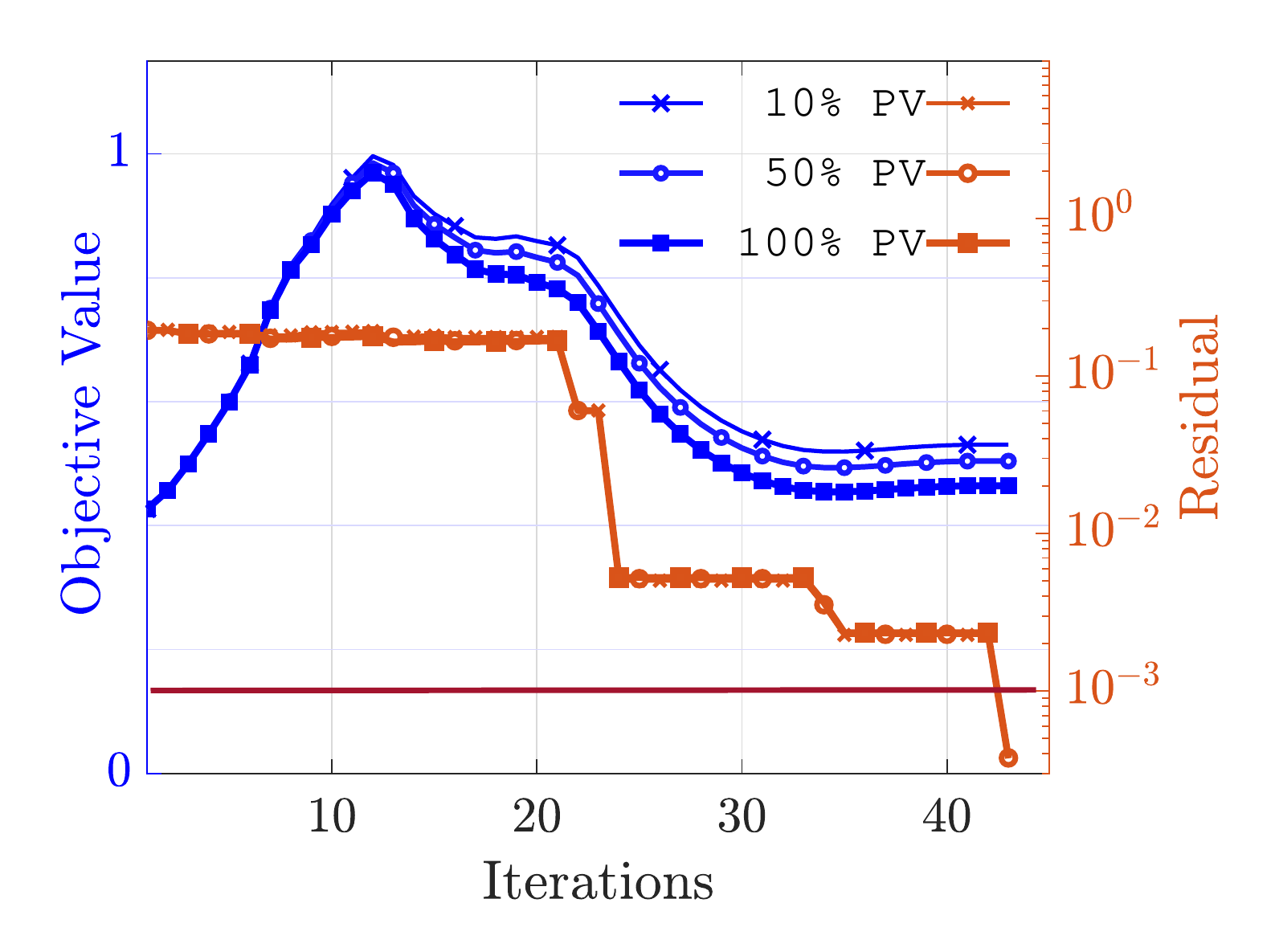}
    \vspace*{-0.6cm}
    \caption{$\D$V Minimization}
  \end{subfigure}
  \vspace*{0.2cm}
  \caption{Max Residual and Objective Value Convergence
  }
  \vspace*{-0.4cm}
  \label{convergence}
\end{figure}

\begin{figure}[ht!]
  \centering
  \begin{subfigure}{0.49\columnwidth}
    \includegraphics[width=\columnwidth]{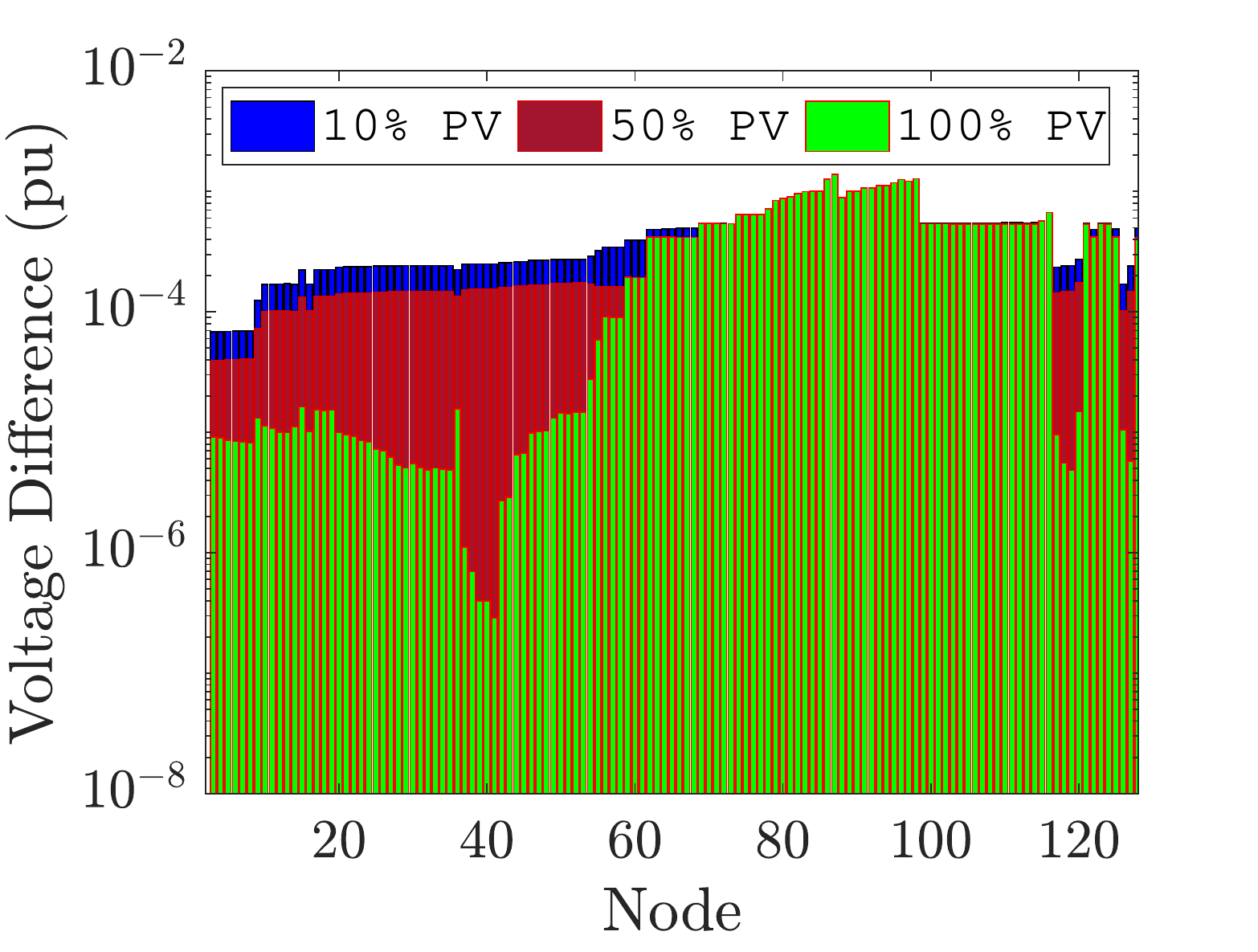}
    \caption{Loss Minimization}
  \end{subfigure}
  \begin{subfigure}{0.49\columnwidth}
    \includegraphics[width=\columnwidth]{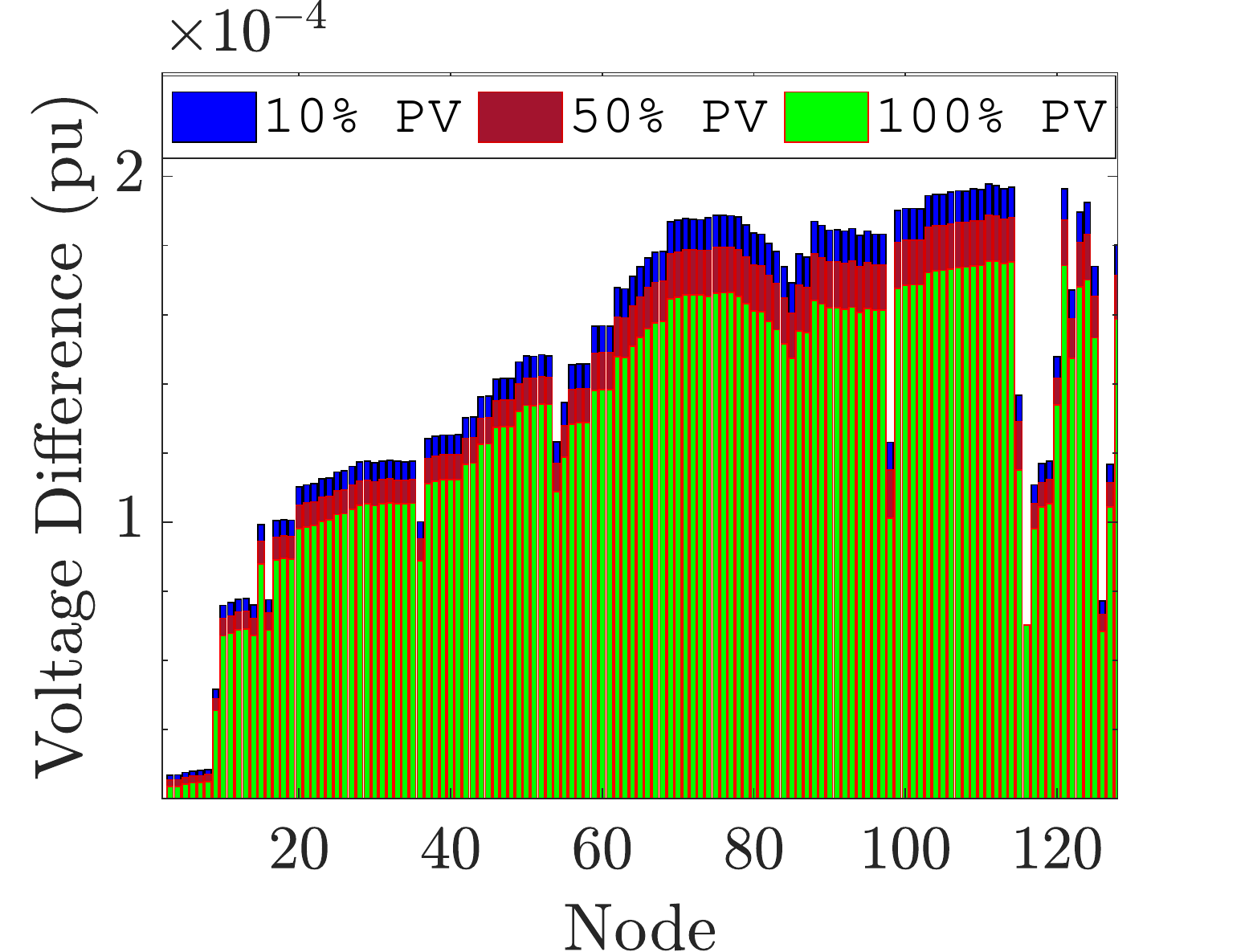}
    \caption{$\D$V Minimization}
  \end{subfigure}
  \caption{Validation: Nodal Voltage Comparison
  }
  \vspace{-0.5cm}
  \label{Validation}
\end{figure}

\begin{figure}[!h]
  \centering
  \begin{subfigure}{0.49\columnwidth}
    \includegraphics[width=\columnwidth]{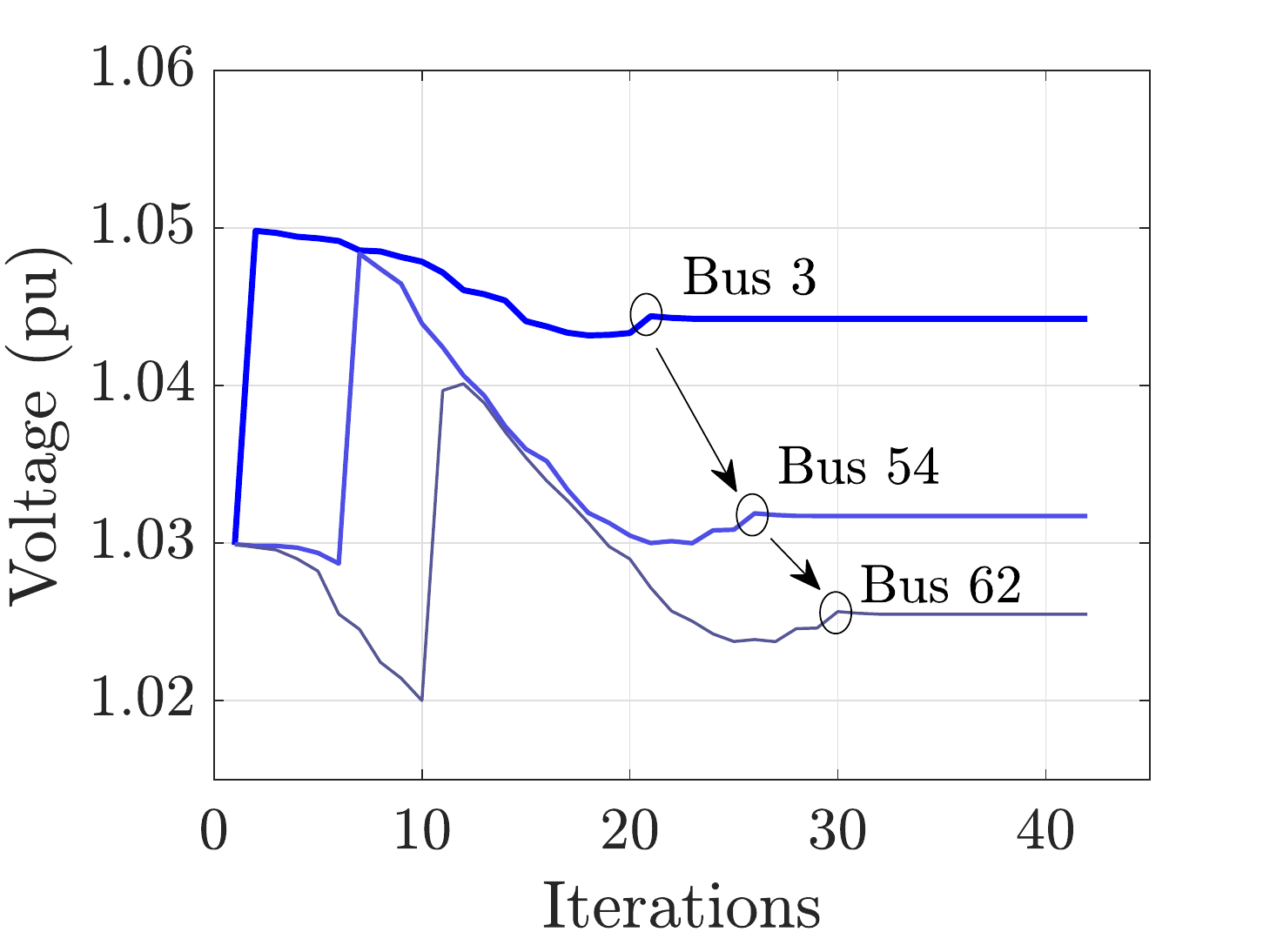}
    \caption{Loss Minimization}
  \end{subfigure}
  \begin{subfigure}{0.49\columnwidth}
    \includegraphics[width=\columnwidth]{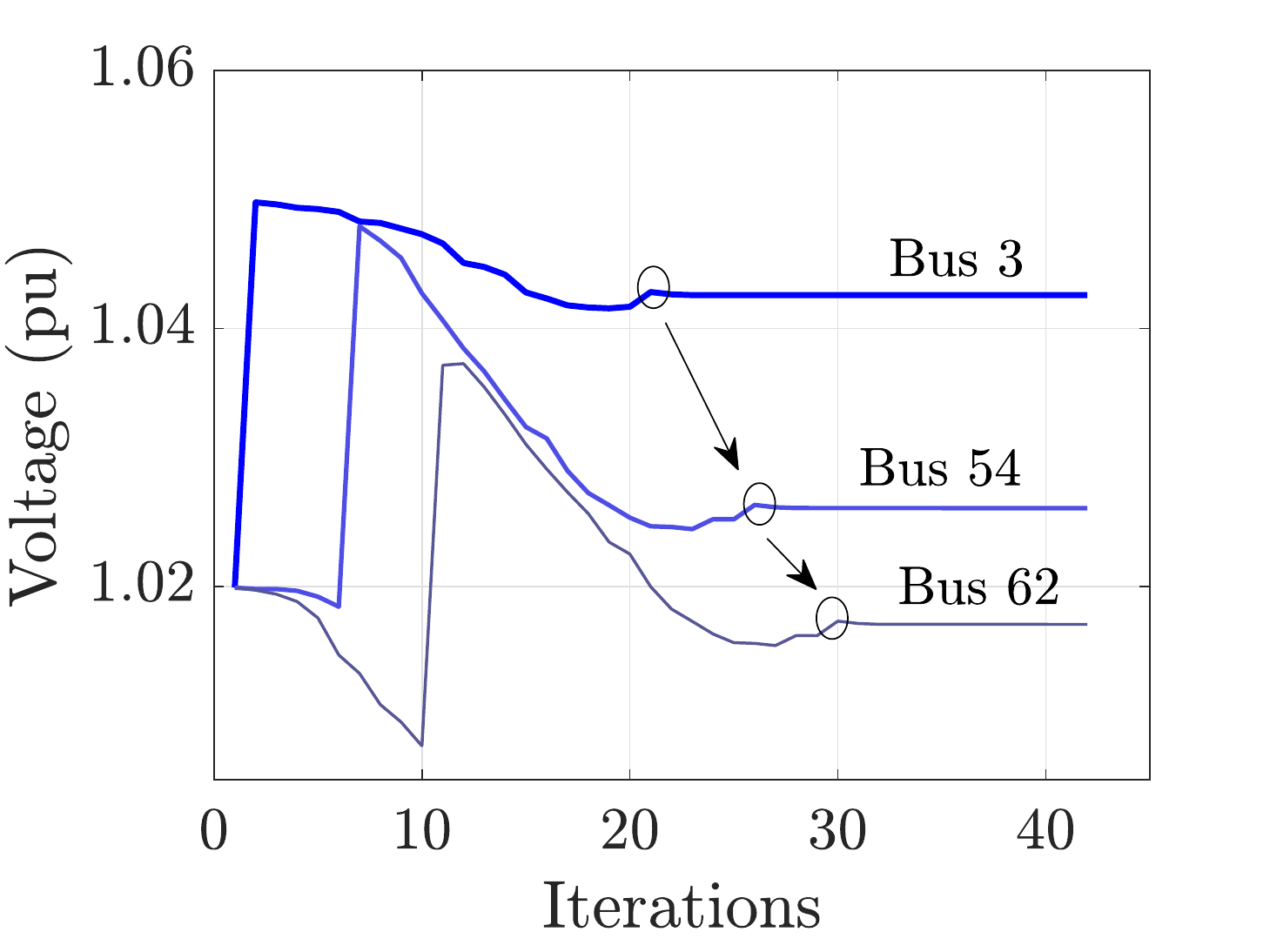}
    \caption{$\D$V Minimization}
  \end{subfigure}
  \caption{Convergence: From Root to Leaf nodes
  }
  \label{Voltage_conv}
  \vspace{-0.5cm}
\end{figure}

\begin{small}
\begin{table}[!b]
    \centering
    \vspace{-0.5cm}
    \caption{{Objective Value Comparison}}
    \vspace{-0.2cm}
    \label{comp_table}
    \setlength{\tabcolsep}{4.5pt}
    \begin{tabular}{|c|c|c|c|c|}
    \hline
    \multirow{1}{*}{\textbf{OPF Problem}}& \multirow{1}{*}{\textbf{Method}}  & \multirow{1}{*}{\textbf{10\% PV}} & \multirow{1}{*}{\textbf{50\% PV}}  & \multirow{1}{*}{\textbf{100\% PV}}\\
    \hline
    \multirow{2}{*}{Loss Min (kW)}& Central & 26.4 & 19.6 & 11.78 \\ \cline{2-5}
      & ENDiCo-OPF  &  26.5 & 19.6 & 11.80\\
    \hline
    \multirow{2}{*}{$\Delta V$ Min (pu)}& Central & 0.5300 & 0.5038 & 0.4640 \\ \cline{2-5}
       & ENDiCo-OPF  & 0.5306 & 0.5042 & 0.4642 \\
    \hline
    \end{tabular}
\end{table}
\end{small}

\begin{figure*}[ht!]
  \centering
  \begin{subfigure}{0.672\columnwidth}
    \includegraphics[width=\columnwidth]{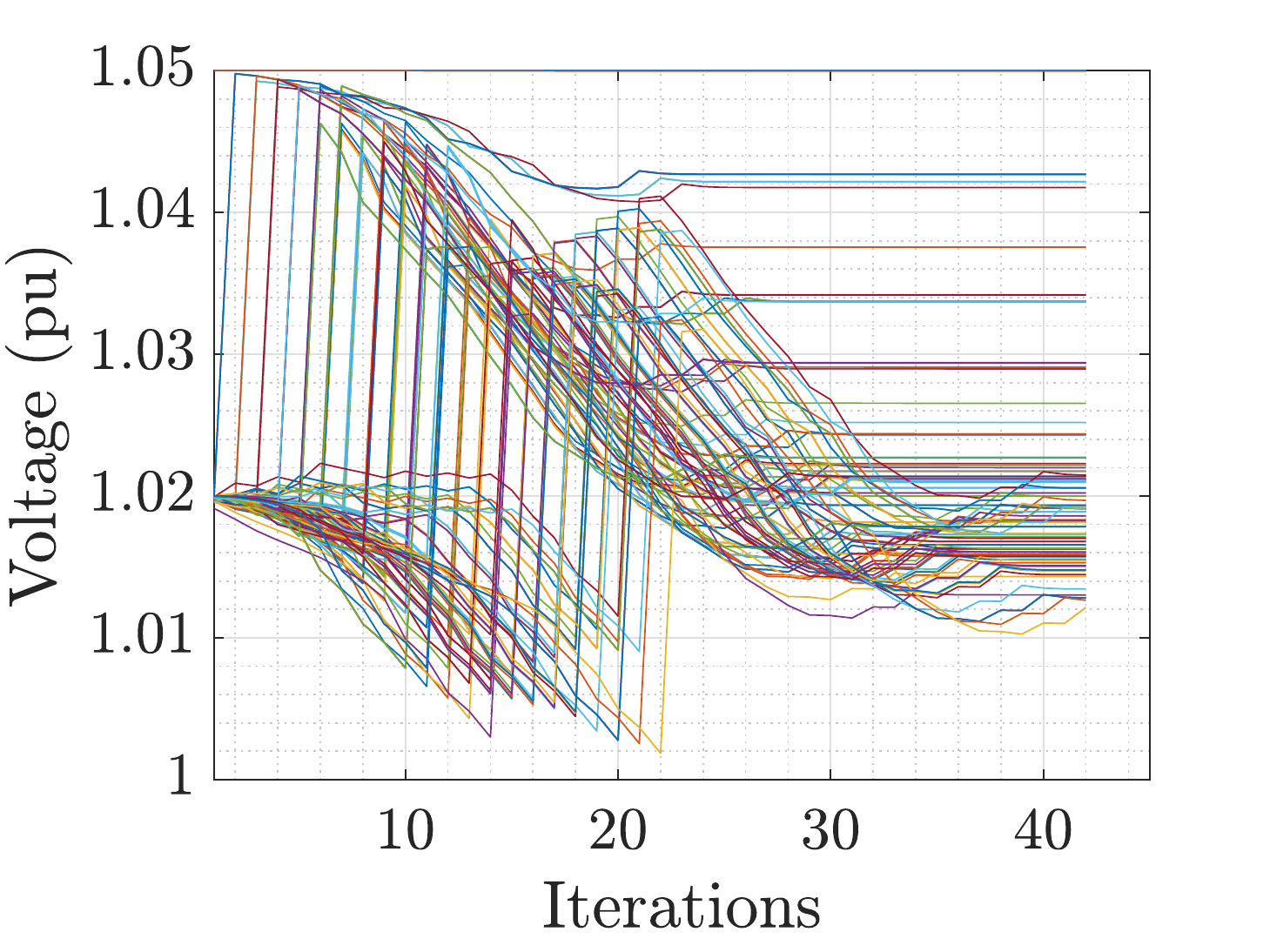}
        \caption{Loss Minimization: 10\% PV Penetration}
  \end{subfigure}
  \begin{subfigure}{0.672\columnwidth}
    \includegraphics[width=\columnwidth]{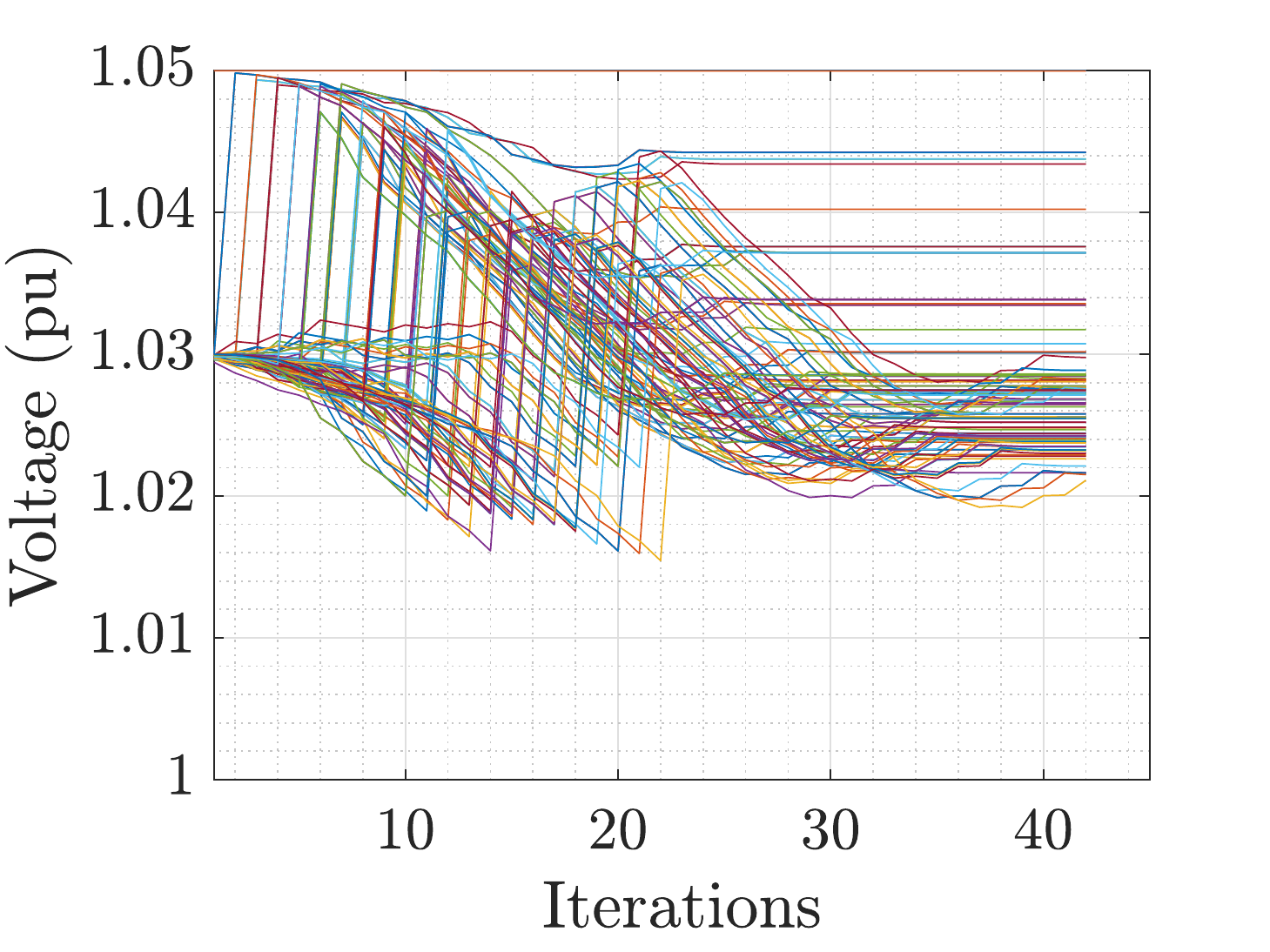}
    \caption{Loss Minimization: 50\% PV Penetration}
  \end{subfigure}
  \begin{subfigure}{0.672\columnwidth}
    \includegraphics[width=\columnwidth]{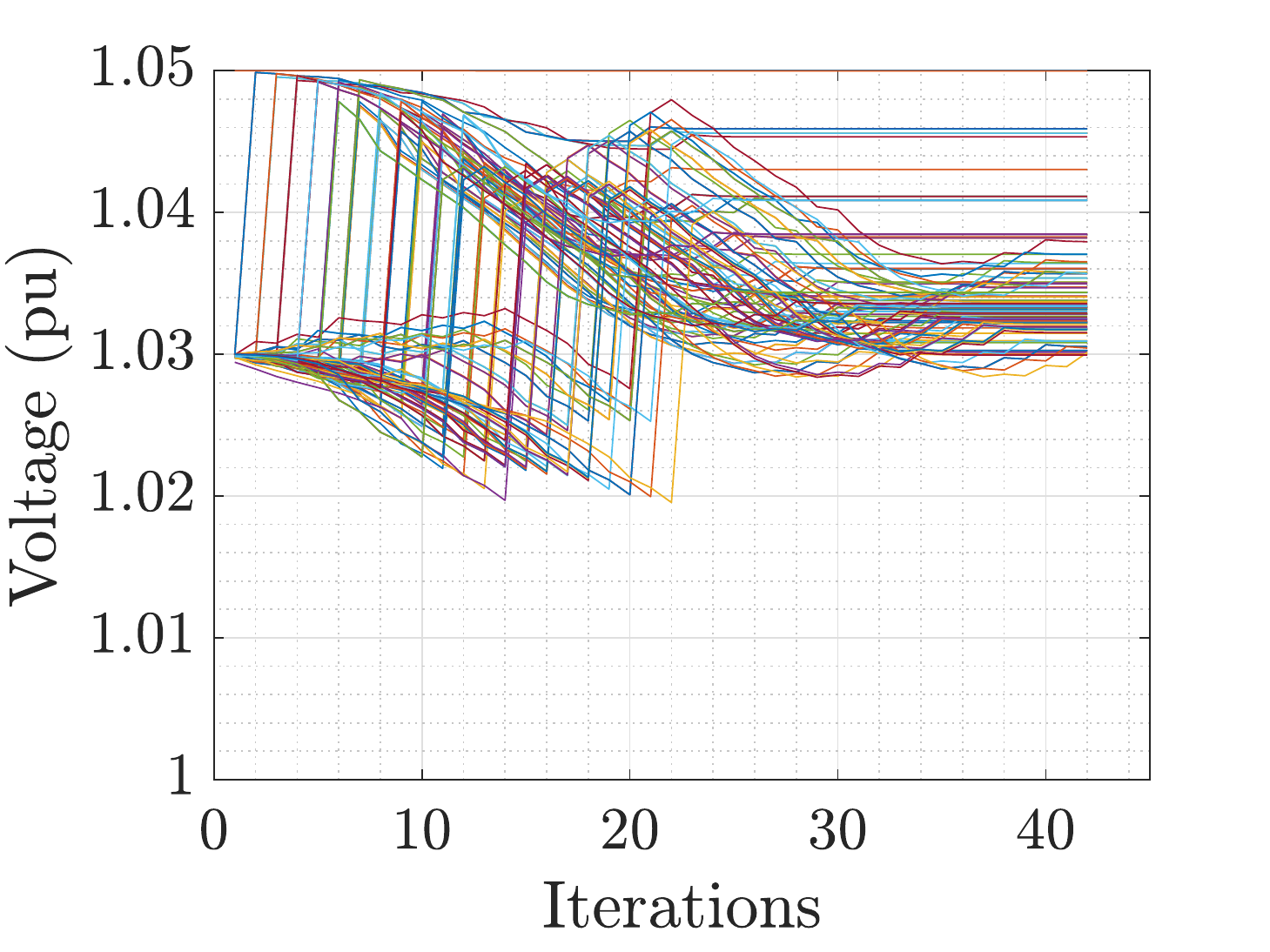}
    \caption{Loss Minimization: 100\% PV Penetration}
  \end{subfigure}
  \begin{subfigure}{0.672\columnwidth}
    \includegraphics[width=\columnwidth]{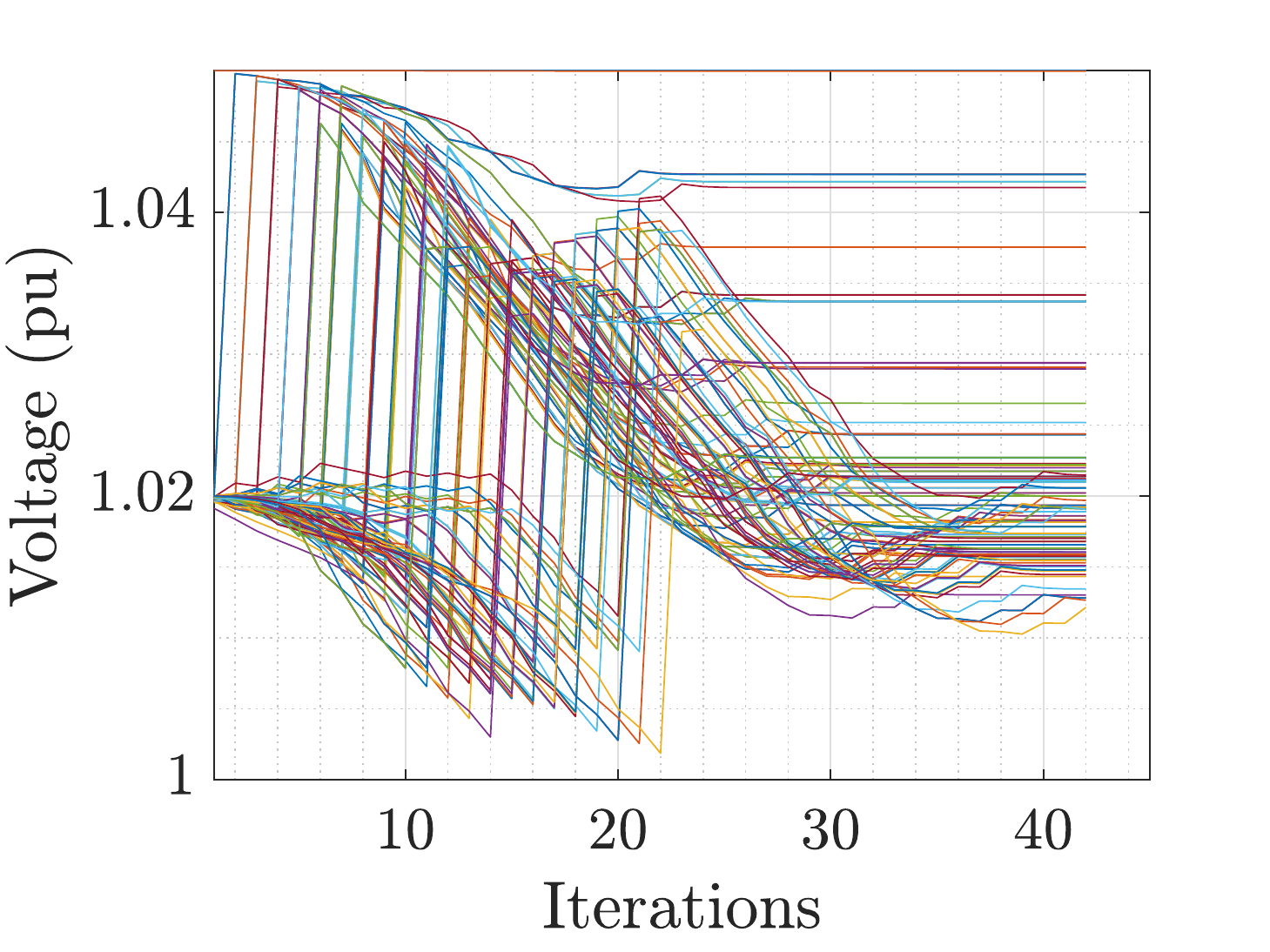}
    \caption{$\D$V Minimization: 10\% PV Penetration}
  \end{subfigure}
  \begin{subfigure}{0.672\columnwidth}
    \includegraphics[width=\columnwidth]{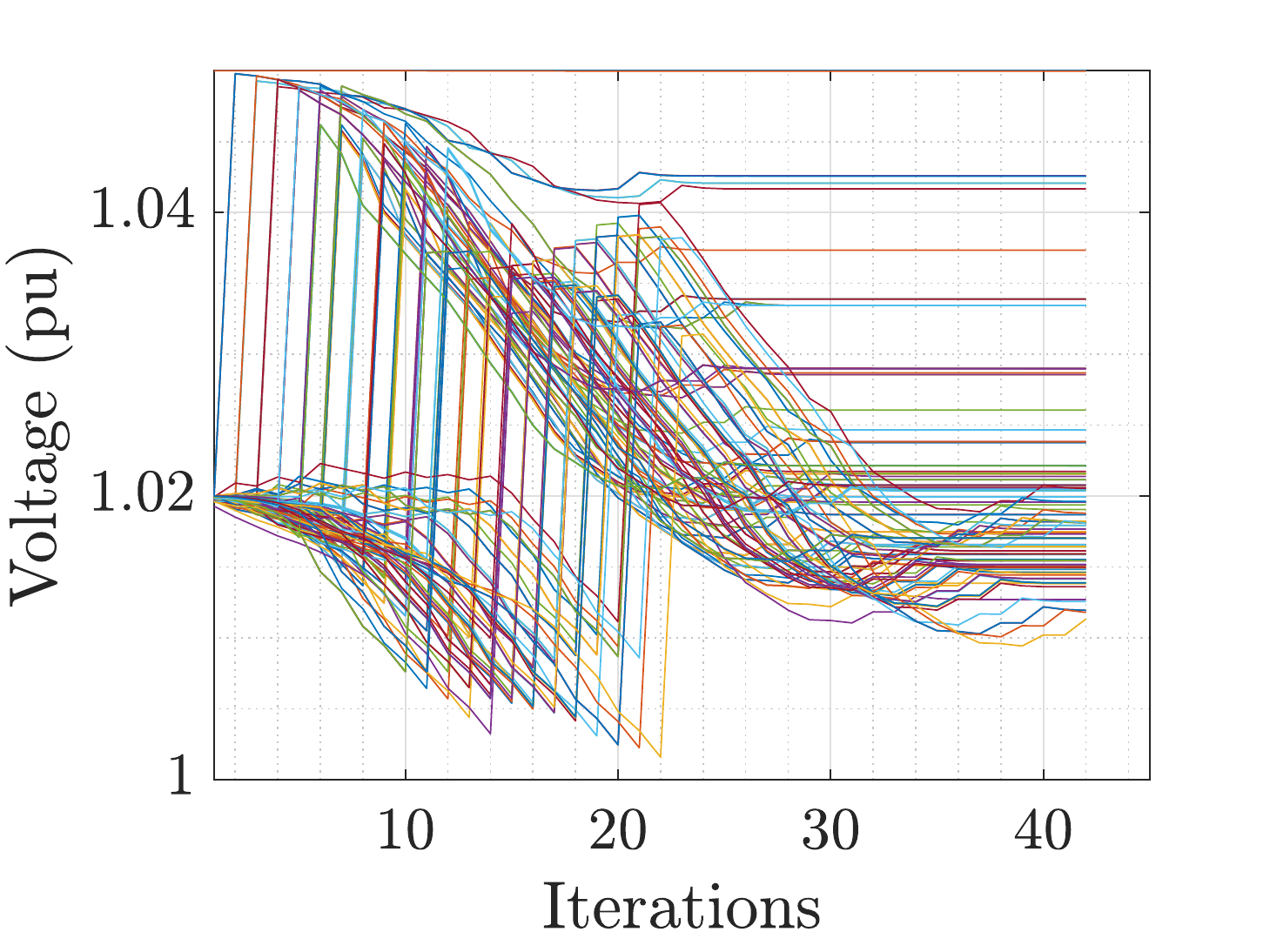}
    \caption{$\D$V Minimization: 50\% PV Penetration}
  \end{subfigure}
  \begin{subfigure}{0.672\columnwidth}
    \includegraphics[width=\columnwidth]{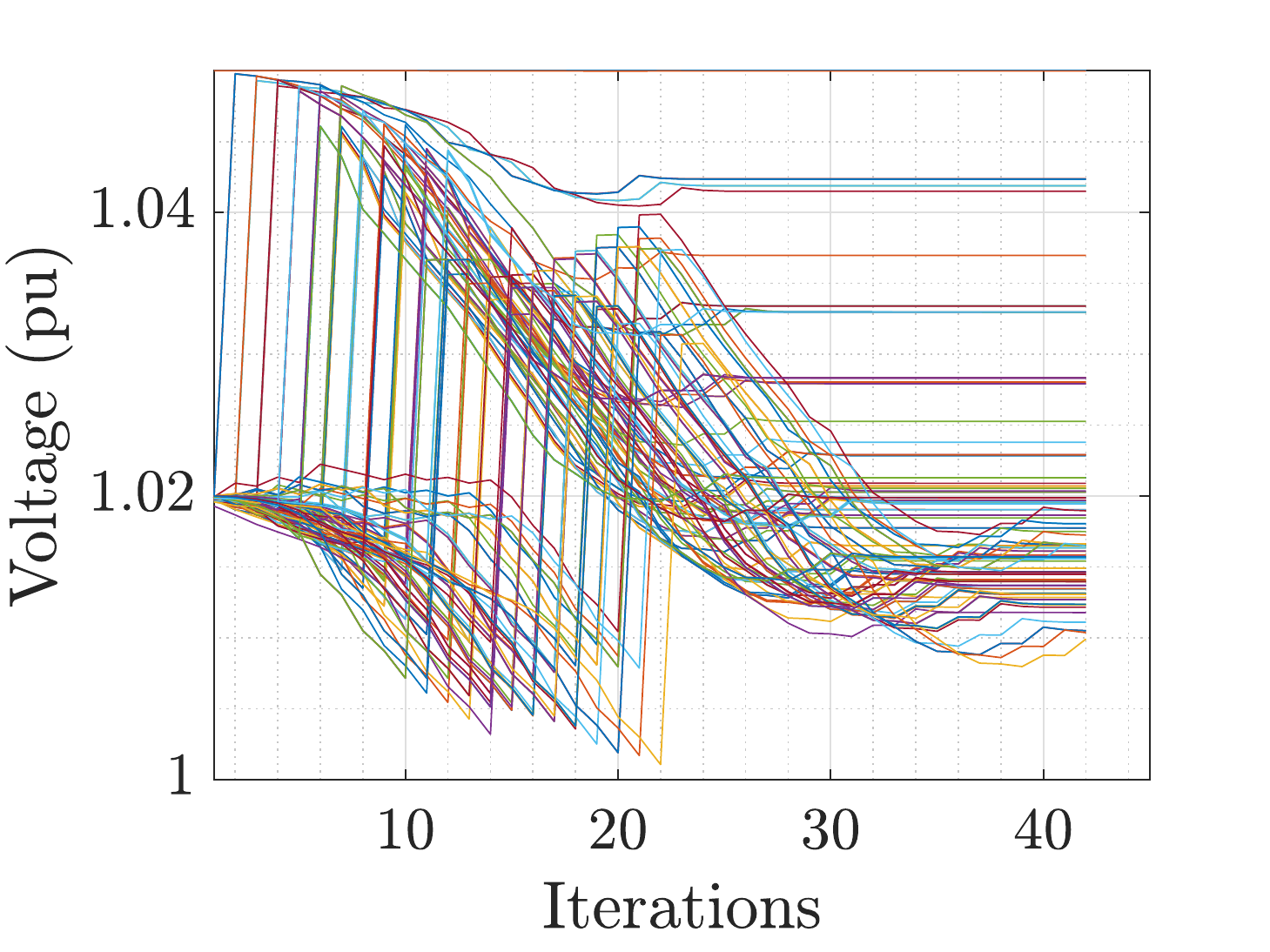}
    \caption{$\D$V Minimization: 100\% PV Penetration}
  \end{subfigure}
  \caption{Boundary Variable Convergence with respect to iterations for Loss (top row) and $\D$V (bottom row) minimizations.} \label{boundV}
\end{figure*}

\subsubsection{Validation of the Optimal Solution}
Besides a faster convergence, we also present the efficacy of the distributed OPF controller in terms of the optimality gaps and feasibility.
To this end, we have compared (a) the objective values,
and (b) the nodal voltages with the centralized solution
(see Fig.~\ref{Validation}).
It can be observed in Table~\ref{comp_table} that the value of the objective functions from centralized and distributed solutions matches for all the cases.
For example, for the 100\% DER penetration case, the line loss is 11.80 kW for proposed ENDiCo-OPF method, and the central solution is 11.78 kW.
Similar comparisons can be found for other DER/PV penetration cases with different OPF objectives.
This validates the solution quality of ENDiCo-OPF.
Further, we can see in Fig.~\ref{Validation} that upon implementing ENDiCo-OPF, the difference in nodal voltages in the system and those from a centralized solution is in the order of $10^{-4}$; this is true for both OPF objectives.
This validates the feasibility of ENDiCo-OPF.

\subsection{Numerical Experiments of Convergence}
In this section, we provide further simulated results on the convergence of the proposed ENDiCo-OPF method.
Here we showcase the boundary variable convergence with respect to iterations, as well as their properties.
We also compare the numerical convergence results with the theoretical analysis presented in Section \ref{sec:cnvgcanal}.

\subsubsection{Convergence at the Boundary} The convergence of the boundary variables (shared boundary voltage) for the simulated cases has been shown in Fig. \ref{Voltage_conv} and Fig.\ref{boundV}; Fig. \ref{Voltage_conv} shows the boundary variables for 100\% PV cases for 3 different locations that helps to visualize the convergence of the shared variable w.r.t. the distance from the root node (substation node).
Specifically, from Fig. \ref{boundV}, we can see that after the initial values, the shared variables (shared nodal voltages) suddenly changes abruptly till 22$^{nd}$ iterations.
This location, i.e., iteration number, where this sudden changes happen depends on the distance of that shared node from the root node (substation node).
For example, `Bus 3' is 2 node distant from the substation, and thus this abrupt changes happen at the 2nd iteration (Fig. \ref{Voltage_conv}); similarly, `Bus 62' is 11 node distant from the substation, and that changes happen at the $11^{th}$ iteration as well.
Along with these characteristics, the overall convergence properties of the shared variables are consistent with both objectives and for all the PV penetration cases as well (Fig. \ref{boundV}).
It corroborates with the statement that the convergence properties is not dependent on the OPF objective or the DER penetration percentage, but rather dependent on the system network.
Instead of using a flat start with $1.02$ pu for the controller, a measured voltage initialization would reduce the iteration number; however, we would like to mention that, this method is robust enough to initialize with any reasonable flat start values.

\subsubsection{Discussion on Convergence}
In \cref{sec:cnvgcanal}, we guaranteed the convergence of the proposed method under some sufficient conditions.
Specifically, we showed (i) convergence of the local sub-problem for a given iteration step (Theorem \ref{thm:subsyscvgce}), (ii) convergence of the local sub-problem over iteration steps (Theorem \ref{thm:single-bd}), and (iii) convergence over time for a line network with multiple nodes (Theorem \ref{thm:globalcvgnc}).
At the same time, our numerical experiments demonstrated similar convergence behavior for more general (than line) networks.

The condition for the convergence of the local sub-problem in a single iteration step is expressed in equation \eqref{eq:maincdn}.
Generally for a stable electric power supply in a power distribution system, $v_i$ and $P_{ij},~Q_{ij}$ are in the order of $1$ pu., and the corresponding line parameters, i.e., $r_{ij},~ x_{ij}$, are both in the order of $10^{-2}$ or less.
This guarantees that condition \eqref{eq:maincdn} is always satisfied for a practical power distribution system.
In our simulated cases, line parameters are also in the order of $\le 10^{-2}$, thus satisfying the condition for Theorem \ref{thm:subsyscvgce}.
Further, this also satisfies most of the sufficient conditions for Theorem \ref{thm:single-bd} and \ref{thm:globalcvgnc}, except the first conditions, i.e., $x_{ij}-r_{ij}\ge 0$ of both of the theorems.
We note that these are sufficient conditions, and that we can observe overall convergence for the cases where $x_{ij}-r_{ij} < 0$ as well.
For the simulation cases, while other sufficient conditions hold true, we observed $x_{ij}-r_{ij} < 0$ for some of the lines, but still the controller converged.
In addition, Fig.~\ref{Voltage_conv} showcases the same result as Theorem \ref{thm:globalcvgnc}.
The node that is closer to the root node/substation node, i.e., the node with a strong voltage source, converges earlier than the node that is more distant from the root node.
For instance, ``Bus 3'', which is two nodes away from the substation, converges earlier than the ``Bus 62'' that is 11 nodes distant from the substation.
Bus 3 converges around the 22nd iteration, whereas Bus 62 converges around 30th iteration for both loss and $\Delta$V minimization optimization problems.

\subsection{Comparison against Centralized OPF and an ADMM-based Distributed OPF Approach}
Table \ref{table:time_comp_table} compares the total solve time for the three different algorithms: a centralized OPF, the proposed distributed ENDiCo-OPF, and an ADMM-based distributed OPF. All three approaches are applied to both objective functions: loss minimization and voltage deviation minimization. The simulation was performed using a Core i7-8550U CPU @ 1.80GHz with 16GB of memory. All three algorithms use fmincon solver from MATLAB to solve the associated nonlinear optimization problems. Since all algorithms use the same compute system and same nonlinear solver, the simulation results provided are appropriate to demonstrate the relative improvements observed via the proposed distributed algorithm. The results show that the proposed distributed approach is significantly faster than both the centralized and ADMM-based distributed OPF methods. For example, for the loss minimization problem with $100\%$ PV penetration, the solution time for ENDiCo-OPF is only $0.71$ seconds, while the centralized OPF and ADMM-based distributed OPF take $15.8$ seconds and $37.3$ minutes, respectively. Moreover, the solution time for the centralized OPF increases with the increase in the number of controllable nodes (i.e. \%PV penetration). The proposed ENDiCo-OPF method, however, scales well even for larger number of controllable nodes. 

\begin{table}[!h]
    \centering
    \caption{Performance Comparison: Solution Time}
    \label{table:time_comp_table}
    \setlength{\tabcolsep}{4.5pt}
    \begin{tabular}{|c|c|c|c|c|}
    \hline
    \multirow{1}{*}{\textbf{OPF Problem}}& \multirow{1}{*}{\textbf{Method}}  & \multirow{1}{*}{\textbf{10\% PV}} & \multirow{1}{*}{\textbf{50\% PV}}  & \multirow{1}{*}{\textbf{100\% PV}}\\
    \hline
    \multirow{3}{*}{Loss Min}& Centralized OPF & 4.2 sec & 10.6 sec & 15.8 sec\\ \cline{2-5}
      & ENDiCo-OPF  &  0.67 sec & 0.71 sec & 0.71 sec\\\cline{2-5}
      & ADMM-based OPF  &  41.6 min & 36.5 min & 37.3 min\\
    \hline
    \multirow{3}{*}{$\Delta V$ Min }& Centralized OPF & 2.1 sec & 3.4 sec & 4.8 sec \\ \cline{2-5}
       & ENDiCo-OPF  & 0.68 sec & 0.70 sec & 0.70 sec \\ \cline{2-5}
       & ADMM-based OPF  &  120 min & 119 min & 121 min\\
    \hline
    \end{tabular}
\end{table}

Additionally, compared to ADMM-based approach, the proposed ENDiCo-OPF method also reduces the required number of communication rounds/iterations by order of magnitudes. Besides the iteration counts, the developed method converges faster compared to the ADMM-based distributed OPF. Figure \ref{fig:comparison_ADMM} illustrates the convergence properties of the objective values for both the ADMM-based method and the ENDiCo-OPF method for $100\%$ PV penetration cases. As can be observed, the ADMM-based method requires $7,000$ and $10,00$ iterations for loss minimization and $\Delta$V minimization problems, respectively. It is worth noting that the proposed ENDiCo-OPF method requires only 42 iterations for both cases, which highlights the effectiveness of the developed method. These results demonstrate that the proposed ENDiCo-OPF method outperforms both centralized OPF and ADMM-based distributed OPF methods.
\begin{figure}[ht!]
  \centering
  \vspace{0.2 cm}
  \begin{subfigure}{0.48\columnwidth}
    \includegraphics[width=\columnwidth]{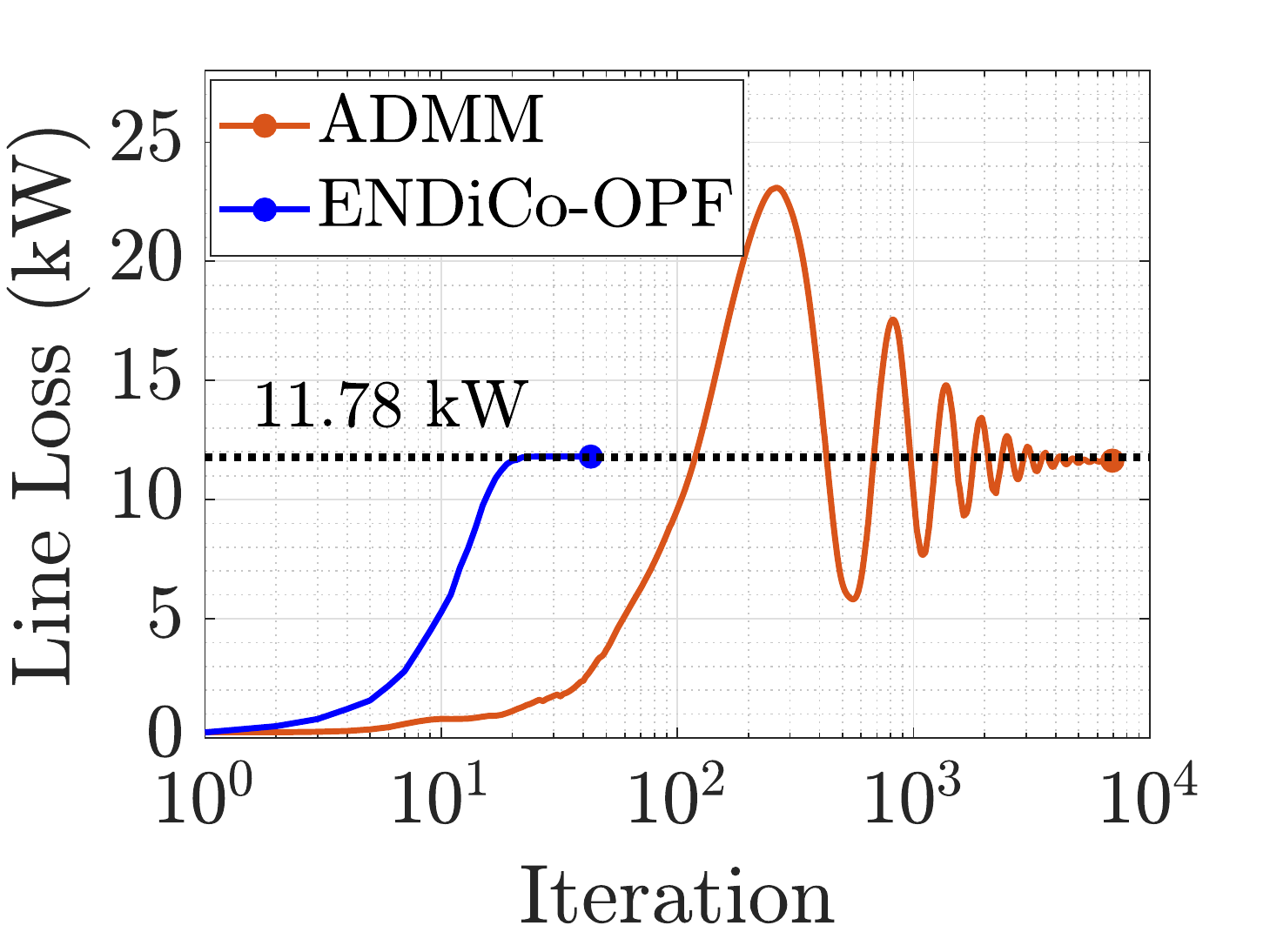}
    \caption{Loss Minimization}
  \end{subfigure}
  \begin{subfigure}{0.48\columnwidth}
    \includegraphics[width=\columnwidth]{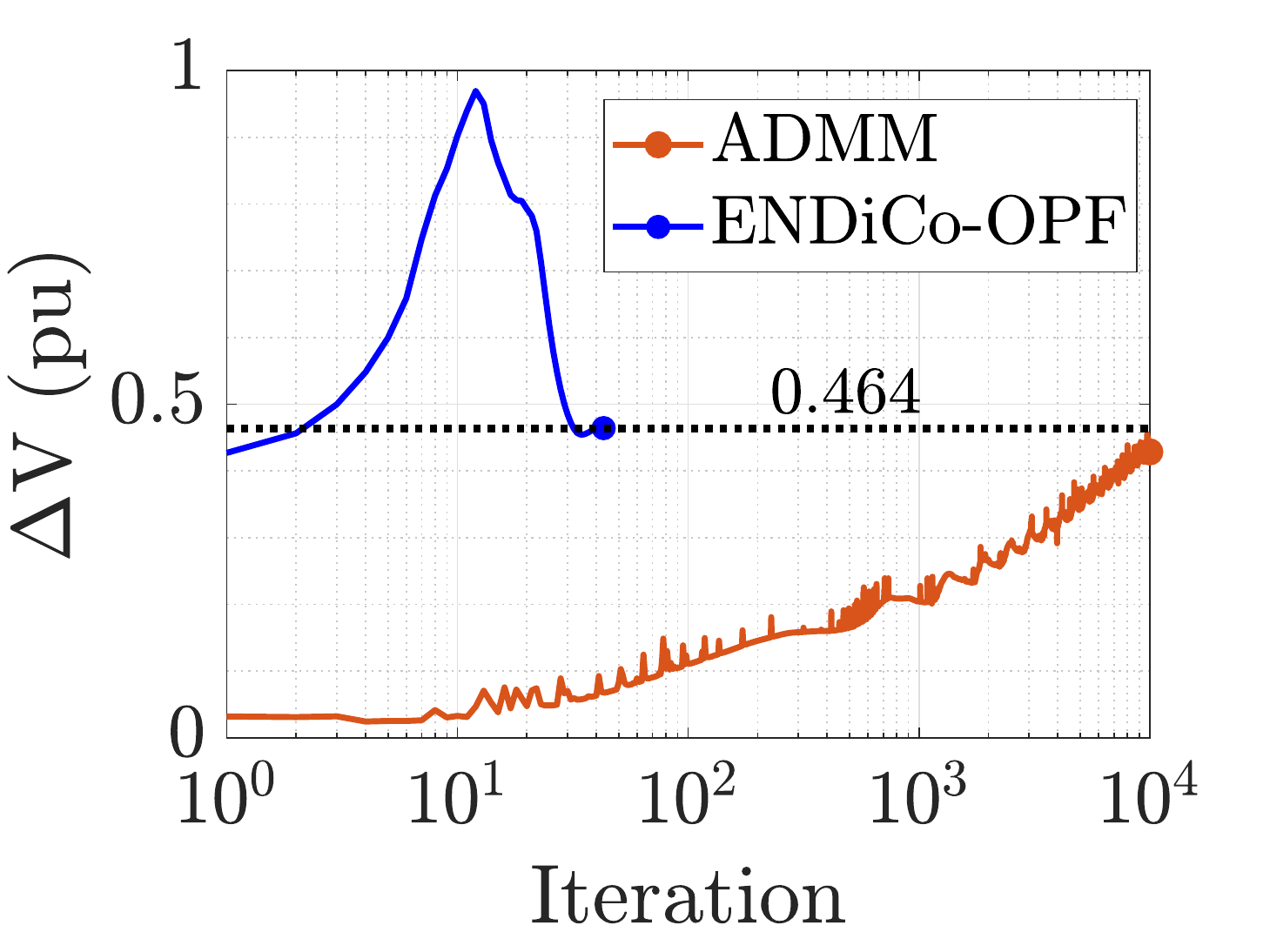}
    \caption{$\Delta$V Minimization}
  \end{subfigure}
  \caption{Comparison of the Convergence of Objective Values.}
  \vspace{-0.5cm}
  \label{fig:comparison_ADMM}
\end{figure}

\subsection{Applicability and Extension to Real-world Setting}
Although in this paper, we focus on a single-period optimization problem, the proposed distributed algorithm has been numerically demonstrated under different realistic test conditions, including for a large-scale single-phase system consisting of over $50,000$ variables \cite{sadnan2023distributed}, three-phase unbalanced systems \cite{S_N_2021distributed_GM}, and simulations conducted under diverse communication conditions \cite{gray2021effects}. We would also like to emphasize that the improvement in computational speed for single-period optimization, observed in this work, will help scale more complex versions of OPF problems, including multi-period and stochastic versions. 

Another major challenge relate to optimization under fast varying conditions. The existing literature, employs online optimization approaches where the algorithm doesn't wait to obtain optimal solution, but rather takes step towards the steepest decent direction \cite{qu2019optimal,bolognani2014distributed}. In our previous work, we have extended the proposed distributed approach to a setting similar to online optimization techniques \cite{sadnan2020real}. Our simulations show that the proposed approach is able to efficiently track the optimal solution, even for rapidly changing system conditions, modeled as fast varying load and PV injections.

A reliable communication system is crucial for the practical viability of the distributed OPF algorithms. The communication system is needed to exchange the boundary variables among distributed agents and arrive at a converged system-level optimal solution. Therefore, the convergence, speed and accuracy of distributed OPF methods depend upon the communication systems conditions. It is imperative to evaluate the impacts of communication system-specific attributes (such as, latency, bandwidth, reliability) on the convergence of distributed OPF algorithm. Related literature includes simplified analysis to numerically evaluate the effects of communication system-related challenges \cite{patari2022distributed}. In our prior work, we have used a cyber-power co-simulation platform, using HELICS, to evaluate several related concerns \cite{gray2021effects}. Additional work is needed to both numerically and analytically evaluate the effects of communication system attributes on the convergence and performance of distributed optimization methods. One can also determine an optimal communication system design to meet the required performance for distributed optimization methods.

\section{Conclusions}
The optimal coordination of growing DER penetrations requires computationally efficient models for distribution-level optimization.
In this paper, we have developed a nonlinear distributed optimal power flow algorithm with convergence guarantees using network equivalence methods.
Then we present sufficient conditions to guarantee the convergence of the proposed method. While our most general sufficient conditions for global convergence over time are presented for line networks, our numerical simulations demonstrate similar convergence behavior for more general, e.g., radial, networks. The numerical simulation on the IEEE 123 bus test system corroborates the theoretical analysis. The proposed distributed method is also validated by comparing the results with a centralized formulation. Developing similar sufficient conditions for global convergence of radial networks, or other general network topologies, will be of high interest.


\input{main.bbltex}

\end{document}